\documentclass[12pt,oneside,english]{amsart}
\usepackage[T1]{fontenc}
\usepackage[latin9]{inputenc}
\usepackage{amsmath}
\usepackage{amsthm}
\usepackage{amssymb}
\usepackage{color}
\usepackage{amssymb,amscd}
\makeatletter
\numberwithin{equation}{section}
\numberwithin{figure}{section}
\theoremstyle{plain}
\newtheorem{thm}{\protect\theoremname}
  \theoremstyle{plain}
  \newtheorem{lem}[thm]{\protect\lemmaname}
  \theoremstyle{remark}
  \newtheorem{rem}[thm]{\protect\remarkname}
  \theoremstyle{definition}
  \newtheorem{defn}[thm]{\protect\definitionname}
  \theoremstyle{plain}
  \newtheorem{conjecture}[thm]{\protect\conjecturename}
  \theoremstyle{remark}
  
  \theoremstyle{definition}
  \newtheorem{example}[thm]{\protect\examplename}
  \theoremstyle{plain}
  \newtheorem{cor}[thm]{\protect\corollaryname}
  \theoremstyle{plain}
  \newtheorem{prop}[thm]{\protect\propositionname}
  \theoremstyle{definition}

\makeatother

\usepackage{babel}
  \providecommand{\claimname}{Claim}
  \providecommand{\conjecturename}{Conjecture}
  \providecommand{\corollaryname}{Corollary}
  \providecommand{\definitionname}{Definition}
  \providecommand{\examplename}{Example}
  \providecommand{\lemmaname}{Lemma}
  \providecommand{\problemname}{Problem}
  \providecommand{\propositionname}{Proposition}
  \providecommand{\remarkname}{Remark}
\providecommand{\theoremname}{Theorem}

\newcommand{\bT}{\mathbf{T}}
\newcommand{\bP}{\mathbf{P}}
\newcommand{\bV}{\mathbf{V}}
\newcommand{\bK}{\mathbf{K}}
\newcommand{\bJ}{\mathbf{J}}
\newcommand{\bX}{\mathbf{X}}
\newcommand{\bY}{\mathbf{Y}}
\newcommand{\bL}{\mathbf{L}}
\newcommand{\bI}{\mathbf{I}}
\newcommand{\bM}{\mathbf{M}}
\newcommand{\bN}{\mathbf{N}}
\newcommand{\bU}{\mathbf{U}}
\newcommand{\bQ}{\mathbf{Q}}
\newcommand{\bW}{\mathbf{\Lambda}}
\newcommand{\bbW}{\mathbf{W}}
\newcommand{\soc}{\operatorname{soc}}

\newcommand{\tens}{\operatorname{tens}}
\newcommand{\Hom}{\operatorname{Hom}}
\newcommand{\Ext}{\operatorname{Ext}}
\newcommand{\fg}{\mathfrak{g}}
\newcommand{\fl}{\mathfrak{l}}
\newcommand{\fa}{\mathfrak{a}}
\newcommand{\fk}{\mathfrak{k}}

\newcommand{\fh}{\mathfrak{h}}
\newcommand{\fs}{\mathfrak{s}}

\newcommand{\Res}{\operatorname{Res}}
\newcommand{\Ker}{\operatorname{ker}}
\begin{document}

\title{Integrable $\mathfrak{sl}(\infty)$-modules and Category $\mathcal O$ for $\mathfrak{gl}(m|n)$}

\author{Crystal Hoyt, Ivan Penkov, Vera Serganova}
\thanks{All three authors have been supported in part by DFG Grant PE 980/6-1. The first and third authors been partially supported by BSF Grant 2012227.  The third author has been also supported by NSF grant DMS-1701532}

\begin{abstract} We introduce and study new categories $\mathbb T_{\fg,\fk}$ of integrable $\fg=\mathfrak{sl}(\infty)$-modules which depend on the choice of a certain reductive in $\fg$ subalgebra $\fk\subset\fg$. The simple objects of $\mathbb T_{\fg,\fk}$ are tensor modules as in the previously studied category  $\mathbb T_{\fg}$ \cite{DPS}; however, the choice of $\fk$ provides for more flexibility of nonsimple modules in $\mathbb T_{\fg,\fk}$ compared to  $\mathbb T_{\fg}$. We then choose $\fk$ to have two infinite-dimensional diagonal blocks, and show that a certain injective object $\bK_{m|n}$ in $\mathbb T_{\fg,\fk}$ realizes a categorical $\mathfrak{sl}(\infty)$-action on the category $\mathcal{O}_{m|n}^{\mathbb{Z}}$, the integral category $\mathcal O$ of the Lie superalgebra $\mathfrak{gl}(m|n)$. We show that the socle of $\bK_{m|n}$ is generated by the projective modules in $\mathcal{O}_{m|n}^{\mathbb{Z}}$, and compute the socle filtration of $\bK_{m|n}$ explicitly. We conjecture that the socle filtration of $\bK_{m|n}$ reflects a ``degree of atypicality filtration'' on the category $\mathcal{O}_{m|n}^{\mathbb{Z}}$. We also conjecture that a natural tensor filtration on $\bK_{m|n}$ arises via the Duflo--Serganova functor sending the category $\mathcal{O}_{m|n}^{\mathbb{Z}}$ to $\mathcal{O}_{m-1|n-1}^{\mathbb{Z}}$. We prove a weaker version of this latter conjecture for the direct summand of  $\bK_{m|n}$ corresponding to finite-dimensional $\mathfrak{gl}(m|n)$-modules.\\

\end{abstract}

\maketitle

\section{Introduction}

Categorification has set a trend in mathematics in the last two decades and has proved important and useful. The opposite process of studying a given category via a combinatorial  or algebraic object such as a single module has also borne ample fruit. An example is Brundan's idea from 2003 to study the category $\mathcal F_{m|n}^{\mathbb Z}$ of finite-dimensional integral modules over the Lie superalgebra $\mathfrak{gl}(m|n)$ via the weight structure of the $\mathfrak{sl}(\infty)$-module $\Lambda^{m}\bV\otimes\Lambda^{n}\bV_*$, where $\bV$ and $\bV_*$ are the two nonisomorphic defining (natural) representations of $\mathfrak{sl}(\infty)$. Using this approach  Brundan computes decomposition numbers in $\mathcal F_{m|n}^{\mathbb Z}$ \cite{B}.   An extension of Brundan's approach was proposed in the work
of Brundan, Losev and Webster in  \cite{BLW}, where a new proof of the Brundan--Kazhdan--Lusztig conjecture for the category $\mathcal O$ over the Lie
superalgebra $\mathfrak{gl}(m|n)$ is given. (The first proof of the Brundan--Kazhdan--Lusztig conjecture for the category $\mathcal O$ over the Lie
superalgebra $\mathfrak{gl}(m|n)$  was given by Cheng, Lam and Wang in \cite{CLW}.) The same approach was also used by  Brundan and Stroppel in \cite{BS4}, where the algebra of endomorphisms of a
projective generator in $\mathcal F_{m|n}^{\mathbb Z}$ is described as a certain diagram algebra and the Koszulity of $\mathcal F_{m|n}^{\mathbb Z}$ is established.

The representation theory of the Lie algebra $\mathfrak{sl}(\infty)$ is of independent interest and has been developing actively also for about two decades. In particular, several categories of $\mathfrak{sl}(\infty)$-modules have been singled out and studied in detail, see \cite{DP, PStyr, DPS, PS, Nam}.

The category $\mathbb T_{\mathfrak{sl}(\infty)}$ from \cite{DPS} has been playing a prominent role: its objects are finite-length submodules of a direct sum of several copies of the tensor algebra $T(\bV\oplus\bV_*)$. In \cite{DPS} it is proved that $\mathbb T_{\mathfrak{sl}(\infty)}$ is a self-dual Koszul category, in \cite{SS} it has been shown that  $\mathbb T_{\mathfrak{sl}(\infty)}$ has a universality property, and in \cite{FPS} $\mathbb T_{\mathfrak{sl}(\infty)}$ has been used to categorify the Boson-Fermion Correspondence.

Our goal in the present paper is to find an appropriate category of $\mathfrak{sl}(\infty)$-modules which contains modules relevant to the representation theory of the Lie superalgebras $\mathfrak{gl}(m|n)$. For this purpose, we introduce and study the categories $\mathbb{T}_{\fg,\fk}$, where $\fg=\mathfrak{sl}(\infty)$ and $\fk$ is a reductive subalgebra of $\fg$ containing the diagonal subalgebra and consisting of finitely many blocks along the diagonal.  The Lie algebra $\fk$ is infinite dimensional and is itself isomorphic to the  commutator subalgebra of a finite direct sum of copies of $\mathfrak{gl}(n)$ (for varying $n$) and copies of $\mathfrak{gl}(\infty)$.
When $\fk=\fg$, this new category coincides with  $\mathbb{T}_{\fg}$. A well-known property of the category $\mathbb{T}_{\fg}$ states that for every $\bM\in\mathbb{T}_{\fg}$, any vector $m\in\bM$ is annihilated by a ``large'' subalgebra $\fg'\subset\fg$, i.e. by an algebra which contains  the commutator subalgebra of the centralizer of a finite-dimensional subalgebra $\fs\subset\fg$. For a general $\fk$ as above, the category $\mathbb{T}_{\fg,\fk}$ has the same simple objects as  $\mathbb{T}_{\fg}$ but requires the following for a nonsimple module $\bM$: the annihilator in $\fk$ of every $m\in\bM$ is a large subalgebra of $\fk$. This makes the nonsimple objects of  $\mathbb{T}_{\fg,\fk}$ more ``flexible'' than in those of  $\mathbb{T}_{\fg}$, the degree of flexibility being governed by $\fk$.

In Section~\ref{sl inf}, we study the category $\mathbb{T}_{\fg,\fk}$ in detail, one of our main results being an explicit computation of the socle filtration of an indecomposable injective object  $\bI^{\boldsymbol{\lambda},{\boldsymbol{\mu}}}$ of $\mathbb{T}_{\fg,\fk}$ (where $\boldsymbol{\lambda}$ and $\boldsymbol{\mu}$ are two Young diagrams), see Theorem~\ref{thm:socinj}. An effect which can be observed here is that with a sufficient increase in the number of infinite blocks of $\fk$, the layers of the socle filtration of $\bI^{\boldsymbol{\lambda},{\boldsymbol{\mu}}}$  grow in a ``self-similar'' manner. This shows that $\mathbb{T}_{\fg,\fk}$ is an intricate extension of the category  $\mathbb{T}_{\fg}$ within the category of all integrable $\fg$-modules.

In Section~\ref{super}, we show that studying the category $\mathbb{T}_{\fg,\fk}$ achieves our goal of improving the understanding of the integral category $\mathcal{O}_{m|n}^{\mathbb{Z}}$ for the Lie superalgebra $\mathfrak{gl}(m|n)$. More precisely, we choose $\fk$ to have two blocks, both of them infinite. Then we show that the category $\mathcal{O}_{m|n}^{\mathbb{Z}}$ is a categorification of an injective object $\bK_{m|n}$ in the category  $\mathbb{T}_{\fg,\fk}$. In order to accomplish this, we exploit the properties of $\mathbb{T}_{\fg,\fk}$ as a category, and not just as a collection of modules.
The object $\bK_{m|n}$ of  $\mathbb{T}_{\fg,\fk}$ can be defined as the complexified reduced Grothendieck group of the category $\mathcal{O}_{m|n}^{\mathbb{Z}}$, endowed with an $\mathfrak{sl}(\infty)$-module structure (categorical action of  $\mathfrak{sl}(\infty)$). For $m,n\geq 1$, $\bK_{m|n}$  is an object of  $\mathbb{T}_{\fg,\fk}$, but not of $\mathbb{T}_{\fg}$.  We prove that the socle of $\bK_{m|n}$ as an  $\mathfrak{sl}(\infty)$-module is the submodule generated by classes of projective $\mathfrak{gl}(m|n)$-modules in  $\mathcal{O}_{m|n}^{\mathbb{Z}}$. Moreover, we conjecture that the socle filtration of $\bK_{m|n}$ (which we already know from Section~\ref{sl inf}) arises from filtering the category $\mathcal{O}_{m|n}^{\mathbb{Z}}$ according to the degree of atypicality of $\mathfrak{gl}(m|n)$-modules. We provide some partial evidence toward this conjecture.

We also show that the category $\mathcal{F}_{m|n}^{\mathbb{Z}}$ of finite-dimensional integral $\mathfrak{gl}(m|n)$-modules categorifies a direct summand $\bJ_{m|n}$ of $\bK_{m|n}$ which is nothing but an injective hull in $\mathbb{T}_{\fg,\fk}$ of Brundan's module  $\Lambda^{m}\bV\otimes\Lambda^{n}\bV_*$, see Corollary ~\ref{fin-dim}.  (Note that the module $\Lambda^{m}\bV\otimes\Lambda^{n}\bV_*$ is an injective object of $\mathbb{T}_{\fg}$, but is not injective in  $\mathbb{T}_{\fg,\fk}$ when $\fk$ has two (or more) infinite blocks.)

Finally, we conjecture that a natural filtration on the category $\mathcal{O}_{m|n}^{\mathbb{Z}}$ defined via the Duflo--Serganova functor $DS:\mathcal{O}_{m|n}^{\mathbb{Z}}\to\mathcal{O}_{m-1|n-1}^{\mathbb{Z}}$ categorifies the tensor filtration of $\bK_{m|n}$, i.e. the coarsest filtration of $\bK_{m|n}$ whose successive quotients are objects of $\mathbb T_{\fg}$. We have a similar conjecture for  the direct summand $\bJ_{m|n}$ of $\bK_{m|n}$, and we provide evidence for this conjecture in Proposition~\ref{thm:module structure}.


\section{Acknowledgements}

We would like to thank two referees for their extremely thorough and thoughtful comments.


\section{New categories of integrable $\mathfrak{sl}(\infty)$-modules}\label{sl inf}

\subsection{Preliminaries}\label{prelims}

Let $\bV$ and $\bV_{*}$ be countable-dimensional vector spaces with fixed bases  $\left\{ v_{i}\right\} _{i\in\mathbb{Z}}$ and  $\left\{ v_{j}^{*}\right\} _{j\in\mathbb{Z}}$, together with a nondegenerate
pairing $\langle\cdot,\cdot\rangle:\bV\otimes \bV_{*}\rightarrow\mathbb{C}$ defined by $\langle v_i,v_j^*\rangle=\delta_{ij}$. Then $\mathfrak{gl}\left(\infty\right):=\bV\otimes \bV_{*}$
has a Lie algebra structure such that
\[
[v_{i}\otimes v_{j}^*,v_{k}\otimes v_{l}^*]=\langle v_{k},v_{j}^*\rangle v_{i}\otimes v_{l}^*-\langle v_{i}, v_{l}^*\rangle v_{k}\otimes v_{j}^*.
\]
We can identify $\mathfrak{gl}(\infty)$ with the
space of infinite matrices $\left(a_{ij}\right)_{i,j\in\mathbb{Z}}$
with finitely many nonzero entries, where the vector $v_{i}\otimes v_{j}^{*}$
corresponds to the matrix $E_{ij}$ with $1$ in the $i,j$-position and zeros elsewhere.
Then $\langle\cdot,\cdot\rangle$ corresponds to the trace map, and its kernel is the Lie algebra  $\mathfrak{sl}\left(\infty\right)$, which is generated by $e_i: = E_{i, i+1}$, $f_i:=E_{i+1,i}$ with $i\in\mathbb{Z}$. One can also
realize $\mathfrak{sl}(\infty)$ as a direct limit of finite-dimensional Lie algebras $\mathfrak{sl}(\infty)=\underrightarrow{\lim}\ \mathfrak{sl}\left(n\right)$.
In contrast to the finite-dimensional setting, the exact sequence
\[
0\to\mathfrak{sl}(\infty)\to\mathfrak{gl}(\infty)\to\mathbb{C}\to0
\]
does not split, and the center of $\mathfrak{gl}(\infty)$ is trivial.

Let $\mathfrak{g}=\mathfrak{sl}(\infty)$.
The representations  $\bV$ and $\bV_{*}$ are the defining representations of $\mathfrak{g}$. The tensor representations $\bV^{\otimes p}\otimes \bV_{*}^{\otimes q}$, $p,q\in\mathbb{Z}_{\geq 0}$ have been studied in \cite{PStyr}. They are not semisimple when $p,q>0$; however, each simple subquotient of  $\bV^{\otimes p}\otimes \bV_{*}^{\otimes q}$ occurs as a submodule of $\bV^{\otimes p'}\otimes \bV_{*}^{\otimes q'}$ for some $p',q'$. The simple submodules of $\bV^{\otimes p}\otimes \bV_{*}^{\otimes q}$ can be parameterized by two Young diagrams $\boldsymbol{\lambda},\boldsymbol{\mu}$,   and we denote them $\bV^{\boldsymbol{\lambda},\boldsymbol{\mu}}$.

Recall that the {\em socle} of a module $\bM$, denoted $\soc\bM$,
is the largest semisimple submodule of $\bM$. The {\em socle filtration} of $\bM$ is defined inductively by $\soc^0 \bM:=\soc\bM$ and $\soc^i \bM:=p_i^{-1}(\soc (\bM/(\soc^{i-1}\bM)))$, where $p_i:\bM\to
\bM/(\soc^{i-1} \bM)$ is the natural projection.
 We also use the notation $\overline{\soc}^i \bM :=\soc^i \bM /\soc^{i-1} \bM $ for the layers of the socle filtration.

Schur-Weyl duality for $\mathfrak{sl}(\infty)$ implies that the module $\bV^{\otimes p}\otimes \bV_{*}^{\otimes q}$ decomposes as
\begin{equation}\label{summands}
\bV^{\otimes p}\otimes \bV_{*}^{\otimes q}=\bigoplus_{|\boldsymbol{\lambda}|=p,|{\boldsymbol{\mu}}|=q}
(\mathbb{S}_{\boldsymbol{\lambda}}(\bV)\otimes \mathbb{S}_{\boldsymbol{{\boldsymbol{\mu}}}}(\bV_*))\otimes (Y_{\boldsymbol{\lambda}}\otimes Y_{{\boldsymbol{\mu}}}),
\end{equation}
where $Y_{\boldsymbol{\lambda}}$ and $Y_{\boldsymbol{\mu}}$ are irreducible $S_p$- and $S_q$-modules, and $\mathbb{S}_{\boldsymbol{{\boldsymbol{\lambda}}}}$ denotes the Schur functor corresponding to the Young diagram (equivalently, partition) $\boldsymbol{\boldsymbol{\lambda}}$.
Each module  $\mathbb{S}_{\boldsymbol{\lambda}}(\bV)\otimes \mathbb{S}_{{\boldsymbol{\mu}}}(\bV_*)$ is indecomposable and its socle filtration is described in \cite{PStyr}. Moreover, Theorem 2.3 of \cite{PStyr} claims that
\begin{equation}\label{socle of S}
\overline{\soc}^k(\mathbb{S}_{\boldsymbol{\lambda}}(\bV)\otimes \mathbb{S}_{{\boldsymbol{\mu}}}(\bV_*))\cong\bigoplus_{\boldsymbol{\lambda'},\boldsymbol{\mu'}, |{\boldsymbol{\gamma}}|=k} N^{\boldsymbol{\lambda}}_{\boldsymbol{\lambda}',{\boldsymbol{\gamma}}}N^{\boldsymbol{\mu}}_{{\boldsymbol{\mu}}',{\boldsymbol{\gamma}}}\bV^{\boldsymbol{\boldsymbol{\lambda}}',{\boldsymbol{\mu}}'}
\end{equation}
where
$N^{\boldsymbol{\lambda}}_{\boldsymbol{\lambda}',{\boldsymbol{\gamma}}}$ are the standard Littlewood-Richardson coefficients. In particular,  $\mathbb{S}_{\boldsymbol{\lambda}}(\bV)\otimes \mathbb{S}_{{\boldsymbol{\mu}}}(\bV_*)$ has simple socle $\bV^{\boldsymbol{\lambda},\boldsymbol{\mu}}$.
It was also shown in  \cite[Theorem 2.2]{PStyr} that the socle of  $\bV^{\otimes p}\otimes \bV_{*}^{\otimes q}$ equals the intersection of the kernels of all contraction maps
\begin{eqnarray}\label{contractions}
&\Phi_{ij}:\bV^{\otimes p}\otimes \bV_{*}^{\otimes q}\rightarrow\bV^{\otimes (p-1)}\otimes \bV_{*}^{\otimes (q-1)} \\
&\nonumber v_1\otimes\cdots\otimes v_p\otimes v_1^*\otimes\cdots\otimes v_q^* \mapsto \langle v^*_j , v_i \rangle v_1\otimes\cdots\otimes \widehat{v_i} \otimes\cdots\otimes v_p\otimes v_1^*\otimes\cdots\otimes \widehat{v_j^*} \otimes\cdots\otimes v_q^*
\end{eqnarray}

A $\fg$-module is called a {\em tensor module} if it is isomorphic to a submodule of a finite direct sum of $\mathfrak{sl}(\infty)$-modules of the form $\bV^{\otimes p_i}\otimes \bV_{*}^{\otimes q_i}$ for $p_i,q_i \in\mathbb{Z}_{\geq 0}$.
The category of tensor modules  $\mathbb{T}_{\mathfrak{g}}$  is by definition the full subcategory of $\mathfrak{g}$-mod consisting of tensor modules \cite{DPS}.
A finite-length $\mathfrak{g}$-module $\bM$ lies in  $\mathbb{T}_{\mathfrak{g}}$ if and only if $\bM$ is integrable and satisfies the large annihilator condition \cite{DPS}. Recall that a $\mathfrak{g}$-module $\bM$ is called  {\em integrable} if  $\mathrm{dim}\{m,x\cdot m,x^2\cdot m,\ldots \}<\infty$ for any $x\in\mathfrak{g}$, $m\in \bM$. A $\mathfrak{g}$-module is said to satisfy the {\em large annihilator condition} if for each $m\in \bM$, the annihilator $\mathrm{Ann}_{\mathfrak{g}}m$ contains the commutator subalgebra of the centralizer of a finite-dimensional subalgebra of $\mathfrak{g}$.

The modules $\bV^{\otimes p}\otimes \bV_{*}^{\otimes q}$, $p,q\in\mathbb{Z}_{\geq 0}$ are injective in the  category $\mathbb{T}_{\mathfrak{g}}$. Moreover, every indecomposable injective object of $\mathbb{T}_{\mathfrak{g}}$ is isomorphic to an indecomposable direct summand of  $\bV^{\otimes p}\otimes \bV_{*}^{\otimes q}$  for some $p,q\in\mathbb{Z}_{\geq 0}$  \cite{DPS}. Consequently, by (\ref{summands}), an indecomposable injective in $\mathbb T_{\fg}$ is isomorphic to
 $\mathbb{S}_{\boldsymbol{\lambda}}(\bV)\otimes \mathbb{S}_{\boldsymbol{\mu}}(\bV_*)$ for some $\boldsymbol{\lambda},\boldsymbol{\mu}$.

The category  $\mathbb{T}_{\mathfrak{g}}$ is a subcategory of the category $\widetilde{Tens}_{\mathfrak{g}}$, which was introduced in \cite{PS} as the full subcategory of $\mathfrak{g}$-mod whose  objects $\bM$ are defined to be the integrable $\mathfrak{g}$-modules of finite Loewy length such that the algebraic dual $\bM^*=\mathrm{Hom}_{\mathbb C}(\bM,\mathbb C)$ is also integrable and of finite Loewy length.
The  categories  $\mathbb{T}_{\mathfrak{g}}$ and  $\widetilde{Tens}_{\mathfrak{g}}$ have the same simple objects $\bV^{\boldsymbol{\lambda},{\boldsymbol{\mu}}}$ \cite{PS,DPS}. The indecomposable  injective objects of  $\widetilde{Tens}_{\mathfrak{g}}$ are (up to isomorphism) the modules $(\bV^{\boldsymbol{\mu},\boldsymbol{\lambda}})^*$, and $\soc (\bV^{\boldsymbol{\mu},\boldsymbol{\lambda}})^*\cong\bV^{\boldsymbol{\lambda},{\boldsymbol{\mu}}}$ \cite{PS}. A recent result of \cite{CP} shows that the Grothendieck envelope  $\overline{Tens}_{\mathfrak{g}}$ of  $\widetilde{Tens}_{\mathfrak{g}}$ is an ordered tensor category, and that any injective object in $\overline{Tens}_{\mathfrak{g}}$  is a direct sum of indecomposable injectives from $\widetilde{Tens}_{\mathfrak{g}}$.


\subsection{The categories $\mathbb{T}_{\fg,\fk}$}

In this section, we introduce new categories of integrable $\mathfrak{sl}(\infty)$-modules. This is motivated in part by the applications to the representation theory of the Lie superalgebras $\mathfrak{gl}(m|n)$.

Let $\fg=\mathfrak{sl}(\infty)$ with the natural representation denoted $\bV$. Consider a decomposition
\begin{equation}\label{decomp V} \bV=\bV_1\oplus\cdots\oplus\bV_r,\end{equation}
for some vector subspaces $\bV_i$ of $\bV$.
Let $\fl$  be the Lie subalgebra of $\fg$ preserving this decomposition. Then $\fk:=[\fl,\fl]$ is isomorphic to
$\fk_1\oplus\cdots\oplus\fk_r$, where each $\fk_i$ is isomorphic to  $\mathfrak{sl}(n_i)$ or $\mathfrak{sl}(\infty)$.

\begin{defn}\label{def: cat T}
Denote by $\widetilde{\mathbb{T}}_{\fg,\fk}$ the full subcategory of $\widetilde{Tens}_{\mathfrak{g}}$ consisting of modules $\bM$ satisfying the large annihilator condition as a module over $\fk_i$ for all $i=1,\dots,r$. By $\mathbb{T}_{\fg,\fk}$ we denote the full subcategory of $\widetilde{\mathbb{T}}_{\fg,\fk}$ consisting of finite-length modules.
\end{defn}

Both categories $\mathbb{T}_{\fg,\fk}$ and  $\widetilde{\mathbb{T}}_{\fg,\fk}$ are abelian symmetric monoidal categories with respect to the usual tensor product of $\fg$-modules. Two categories  $\widetilde{\mathbb{T}}_{\fg,\fk}$ and  $\widetilde{\mathbb{T}}_{\fg,\bar{\fk}}$ are equal if  $\fk$ and $\bar{\fk}$ have finite corank in $\fk+\bar{\fk}$, so we will henceforth assume without loss of generality that each $\bV_i$ in decomposition (\ref{decomp V}) is infinite dimensional. Note that $\mathbb{T}_{\fg,\fg}=\mathbb{T}_{\fg}$.

We define the functor $\Gamma_{\fg,\fk}:\widetilde{Tens}_{\mathfrak{g}}\to \widetilde{\mathbb{T}}_{\fg,\fk}$ by taking the maximal submodule lying in
$\widetilde{\mathbb{T}}_{\fg,\fk}$. Then
\begin{equation}\label{Gamma functor}
\Gamma_{\fg,\fk}(\bM)=\bigcup \bM^{\fs_1\oplus\dots\oplus\fs_r},
\end{equation}
where the union is taken over all finite corank subalgebras $\fs_1\subset\fk_1,\dots,\fs_r\subset\fk_r$.

\begin{lem}\label{lem: R1} Let $\mathbb{T}_{\fg,\fk}$  be as in Definition~\ref{def: cat T}.
\begin{enumerate}
\item The simple objects of $\mathbb{T}_{\fg,\fk}$ and of $\widetilde{\mathbb{T}}_{\fg,\fk}$ are isomorphic to $\bV^{\boldsymbol{\lambda},{\boldsymbol{\mu}}}$.
\item The functor $\Gamma_{\fg,\fk}$  sends injective modules in $\widetilde{Tens}_{\mathfrak{g}}$ to injective modules in $ \widetilde{\mathbb{T}}_{\fg,\fk}$.
\item The category $\widetilde{\mathbb{T}}_{\fg,\fk}$ has  enough injective modules.
\item The indecomposable injective objects of $\widetilde{\mathbb{T}}_{\fg,\fk}$  are isomorphic to $\Gamma_{\fg,\fk}((\bV^{{\boldsymbol{\mu}},\boldsymbol{\lambda}} )^*)$.
\end{enumerate}
\end{lem}
\begin{proof}
\begin{enumerate}
\item The category  $\mathbb{T}_{\fg}$ is a full subcategory of  $\mathbb{T}_{\fg,\fk}$ and of  $\widetilde{\mathbb{T}}_{\fg,\fk}$, which are both full subcategories of  $\widetilde{Tens}_{\mathfrak{g}}$. Since the categories $\mathbb{T}_{\fg}$ and $\widetilde{Tens}_{\mathfrak{g}}$ have the same simple objects $\bV^{\boldsymbol{\lambda},{\boldsymbol{\mu}}}$, the claim follows.
\item This follows from the definition of $\Gamma_{\fg,\fk}$, since $\mathrm{Hom}_{\mathbb{T}_{\fg,\fk}}(X,\Gamma_{\fg,\fk}(Y))= \mathrm{Hom}_{\widetilde{Tens}_{\mathfrak{g}}}(X,Y)$ for all $X\in\mathbb{T}_{\fg,\fk}$ and $Y\in \widetilde{Tens}_{\mathfrak{g}}$.
\item Every module $\bM$ in  $\widetilde{\mathbb{T}}_{\fg,\fk}$ can be embedded into $\Gamma_{\fg,\fk}(\bM^{**})$, which is  injective in  $\widetilde{\mathbb{T}}_{\fg,\fk}$, since $\bM^{**}$ is injective in $\widetilde{Tens}_{\mathfrak{g}}$ \cite{PS}.
\item This follows from (1) and (2), since $(\bV^{{\boldsymbol{\mu}},\boldsymbol{\lambda}})^*$ is an indecomposable injective object of $\widetilde{Tens}_{\mathfrak{g}}$, and consequently $\Gamma_{\fg,\fk}((\bV^{{\boldsymbol{\mu}},\boldsymbol{\lambda}})^*)$ is an indecomposable injective object of $\widetilde{\mathbb{T}}_{\fg,\fk}$ with $\soc \Gamma_{\fg,\fk}((\bV^{{\boldsymbol{\mu}},\boldsymbol{\lambda}} )^*)\cong \bV^{{\boldsymbol{\lambda}},\boldsymbol{\mu}}$.
\end{enumerate}
\end{proof}

\begin{rem}
It will follow from Corollary~\ref{I finte length} that the indecomposable injective objects $\Gamma_{\fg,\fk}((\bV^{{\boldsymbol{\mu}},\boldsymbol{\lambda}})^*)$ are objects of $\mathbb{T}_{\fg,\fk}$. Consequently,
 $\mathbb{T}_{\fg,\fk}$ and  $\widetilde{\mathbb{T}}_{\fg,\fk}$ have the same indecomposable injectives.
\end{rem}


\subsection{The functor $\mathrm{R}$ and Jordan-H\"older multiplicities}

In this section, we calculate the Jordan-H\"older multiplicities of the indecomposable injective objects of the categories $\mathbb{T}_{\fg,\fk}$. One of the main tools we use for this computation is the functor $\mathrm{R}$, which we will now introduce.

Let
\begin{equation}\label{decomp Vr}
\bV'=\bV_1\oplus\cdots\oplus\bV_{r-1}, \quad \fg'=\fg\cap \mathfrak{gl}(\bV'),\quad\fk'=\fk_1\oplus\cdots\oplus\fk_{r-1}.
\end{equation}
 Let $(\bV_r)_*\subset \bV_*$ be the annihilator of $\bV'=\bV_1\oplus\cdots\oplus\bV_{r-1}$ with respect to the pairing $\langle\cdot,\cdot\rangle$.
 We have $\fg'\cong\mathfrak{sl}(\infty)$ and $\fk'\subset\fg'$.

Define a functor $\mathrm{R}$ from the category $\fg$--mod of all $\fg$-modules to the category $\fg'$--mod by setting
$$\mathrm{R}(\bM)=\bM^{\fk_r}.$$
It follows from the definition that after restricting to $ \widetilde{\mathbb T}_{\fg,\fk}$ we have a functor
$\mathrm{R}: \widetilde{\mathbb T}_{\fg,\fk}\to  \widetilde{\mathbb T}_{\fg',\fk'}$.

\begin{lem}\label{lem:commutative} The following diagram of functors is commutative:
  $$\begin{CD}
\mathfrak{g}\mathrm{-mod} @>\mathrm{R}>>{\mathfrak{g'}}\mathrm{-mod}\\
@V\Gamma_{\fg,\fk}VV @V\Gamma_{\fg',\fk'}VV \\
\widetilde{\mathbb{T}}_{\fg,\fk} @>\mathrm{R}>>\widetilde{ \mathbb{T}}_{\fg',\fk'}
\end{CD}.\notag$$
\end{lem}
\begin{proof} By (\ref{Gamma functor}) we have
  $$\Gamma_{\fg,\fk}(\bM)=\bigcup \bM^{\fs_1\oplus\dots\oplus\fs_r}$$
 for any $\fg$-module $\bM$. Then
$$\mathrm{R}(\Gamma_{\fg,\fk}(\bM))=(\bigcup \bM^{\fs_1\oplus\dots\oplus\fs_r})^{\fk_r}=\bigcup \bM^{\fs_1\oplus\dots\oplus\fs_{r-1}\oplus\fk_r}=\bigcup (\mathrm{R}(\bM))^{\fs_1\oplus\dots\oplus\fs_{r-1}}=
\Gamma_{\fg',\fk'}(\mathrm{R}(\bM)).$$
  \end{proof}

  \begin{lem}\label{lem:identity} If $\boldsymbol{\lambda},\boldsymbol{\mu}$ are Young diagrams, then
    $$\mathrm{R}((\mathbb{S}_{\boldsymbol{\lambda}}(\bV)\otimes \mathbb{S}_{{\boldsymbol{\mu}}}(\bV_*))^*)=\bigoplus_{\boldsymbol{\lambda'},\boldsymbol{\mu'},{\boldsymbol{\gamma}}} N^{\boldsymbol{\lambda}}_{\boldsymbol{\lambda}',{\boldsymbol{\gamma}}}N^{{\boldsymbol{\mu}}}_{{\boldsymbol{\mu}}',{\boldsymbol{\gamma}}}(\mathbb{S}_{\boldsymbol{\lambda}'}(\mathrm{R}(\bV))\otimes \mathbb{S}_{{\boldsymbol{\mu}}'}(\mathrm{R}(\bV_*)))^*.$$
  \end{lem}
  \begin{proof} Since $\mathrm{R}(\bV)=\bV'$, we have the decompositions
    $$\bV=\mathrm{R}(\bV)\oplus\bV_r,\quad\bV_*=\mathrm{R}(\bV_*)\oplus(\bV_r)_*.$$
    We also have the identity
    \begin{equation}\label{Schur}
      \mathbb{S}_{\boldsymbol{\lambda}}(V\oplus W)=\bigoplus N^{\boldsymbol{\lambda}}_{{\boldsymbol{\mu}},{\boldsymbol{\nu}}}\mathbb{S}_{\boldsymbol{\mu}}(V)\otimes \mathbb{S}_{{\boldsymbol{\nu}}}(W), \end{equation}
    which holds for all vector spaces $V$ and $W$.
     These imply
    $$\mathbb{S}_{\boldsymbol{\lambda}}(\bV)\otimes \mathbb{S}_{{\boldsymbol{\mu}}}(\bV_*)=\bigoplus_{\boldsymbol{\lambda'},\boldsymbol{\mu'},{\boldsymbol{\gamma}},{\boldsymbol{\gamma}}'} N^{\boldsymbol{\lambda}}_{\boldsymbol{\lambda}',{\boldsymbol{\gamma}}}N^{{\boldsymbol{\mu}}}_{{\boldsymbol{\mu}}',{\boldsymbol{\gamma}}'}\mathbb{S}_{\boldsymbol{\lambda}'}(\mathrm{R}(\bV))\otimes
    \mathbb{S}_{{\boldsymbol{\gamma}}}(\bV_r)\otimes \mathbb{S}_{{\boldsymbol{\mu}}'}(\mathrm{R}(\bV_*))\otimes \mathbb{S}_{{\boldsymbol{\gamma}}'}((\bV_r)_*).$$

 By definition
    $$\mathrm{R}((\mathbb{S}_{\boldsymbol{\lambda}}(\bV)\otimes \mathbb{S}_{{\boldsymbol{\mu}}}(\bV_*))^*)=\operatorname{Hom}_{\fg'}(\mathbb{S}_{\boldsymbol{\lambda}}(\bV)\otimes \mathbb{S}_{{\boldsymbol{\mu}}}(\bV_*),\mathbb C),$$
  and it follows from (\ref{socle of S}) that
    $$\dim\Hom_{\fg'}(\mathbb{S}_{{\boldsymbol{\gamma}}}(\bV_r)\otimes  \mathbb{S}_{{\boldsymbol{\gamma}}'}((\bV_r)_*),\mathbb C)=\delta_{{\boldsymbol{\gamma}},{\boldsymbol{\gamma}}'},$$
$\delta_{{\boldsymbol{\gamma}},{\boldsymbol{\gamma}}'}$ being Kronecker's delta.
 Therefore,
    $$\operatorname{Hom}_{\fg'}(\mathbb{S}_{\boldsymbol{\lambda}}(\bV)\otimes \mathbb{S}_{{\boldsymbol{\mu}}}(\bV_*),\mathbb C)=\bigoplus_{\boldsymbol{\lambda'},\boldsymbol{\mu'},{\boldsymbol{\gamma}}} N^{\boldsymbol{\lambda}}_{\boldsymbol{\lambda}',{\boldsymbol{\gamma}}}N^{{\boldsymbol{\mu}}}_{{\boldsymbol{\mu}}',{\boldsymbol{\gamma}}}(\mathbb{S}_{\boldsymbol{\lambda}'}(\mathrm{R}(\bV))\otimes \mathbb{S}_{{\boldsymbol{\mu}}'}(\mathrm{R}(\bV_*)))^*.$$
    \end{proof}

    \begin{lem}\label{lem:dual} If $0\to A\to B\to C\to 0$ is an exact sequence of modules in $\widetilde{Tens}_{\fg}$, then the dual exact sequence $0\to C^*\to B^*\to A^*\to 0$ splits.
    \end{lem}
    \begin{proof} This follows from the fact that $C^*$ is injective in $\widetilde{Tens}_{\fg}$.
      \end{proof}

      \begin{lem}\label{lem:injective} The functor $\mathrm{R}: \widetilde{\mathbb T}_{\fg,\fk}\to  \widetilde{\mathbb T}_{\fg',\fk'}$ sends an  indecomposable  injective object to an injective object.
\end{lem}
\begin{proof} Let $\bP^{\boldsymbol{\lambda},\boldsymbol{\mu}}=\Gamma_{\fg,\fk}((\mathbb{S}_{\boldsymbol{\lambda}}(\bV)\otimes \mathbb{S}_{{\boldsymbol{\mu}}}(\bV_*))^*)$. Then by Lemma \ref{lem:commutative} we have
  $$\mathrm{R}(\bP^{\boldsymbol{\lambda},\boldsymbol{\mu}})=\Gamma_{\fg',\fk'}(\mathrm{R}((\mathbb{S}_{\boldsymbol{\lambda}}(\bV)\otimes \mathbb{S}_{{\boldsymbol{\mu}}}(\bV_*))^*)),$$  and hence by Lemma \ref{lem:identity}
  \begin{equation} \label{formula RI} \mathrm{R}(\bP^{\boldsymbol{\lambda},\boldsymbol{\mu}})=\bigoplus_{\boldsymbol{\lambda'},\boldsymbol{\mu'},{\boldsymbol{\gamma}}} N^{\boldsymbol{\lambda}}_{\boldsymbol{\lambda}',{\boldsymbol{\gamma}}}N^{{\boldsymbol{\mu}}}_{{\boldsymbol{\mu}}',{\boldsymbol{\gamma}}}\Gamma_{\fg',\fk'}((\mathbb{S}_{\boldsymbol{\lambda}'}(\mathrm{R}(\bV))\otimes \mathbb{S}_{{\boldsymbol{\mu}}'}(\mathrm{R}(\bV_*)))^*).\end{equation}
  Therefore, $\mathrm{R}(\bP^{\boldsymbol{\lambda},\boldsymbol{\mu}})$ is injective in $\widetilde{\mathbb T}_{\fg',\fk'}$. Every indecomposable injective object in  $\widetilde{\mathbb T}_{\fg,\fk}$ is isomorphic to $\Gamma_{\fg,\fk}(\bL^*)$
  for some simple object   $\bL=\bV^{\boldsymbol{\lambda},{\boldsymbol{\mu}}}$, and by  Lemma \ref{lem:dual},  $\Gamma_{\fg,\fk}(\bL^*)$ is a  direct summand of $\bP^{\boldsymbol{\lambda},\boldsymbol{\mu}}=\Gamma_{\fg,\fk}((\mathbb{S}_{\boldsymbol{\lambda}}(\bV)\otimes \mathbb{S}_{{\boldsymbol{\mu}}}(\bV_*))^*)$. Since the functor $\mathrm{R}$ is left exact,  $\mathrm{R}(\Gamma_{\fg,\fk}(\bL^*))$ is a direct summand of $\mathrm{R}(\bP^{\boldsymbol{\lambda},\boldsymbol{\mu}})$.  Hence, $\mathrm{R}(\Gamma_{\fg,\fk}(\bL^*))$ is injective in  $\widetilde{\mathbb T}_{\fg',\fk'}$.
\end{proof}
\begin{lem}\label{aux0}  Let $\bV=V_n\oplus \bbW$ and $\bV_*=V_n^*\oplus\bbW_*$ be decompositions with $\dim V_n=n$,   $\bbW^\perp=V_n^*$ and $\bbW_*^\perp=V_n$.
Let $\fs$ be the commutator subalgebra of  $\bbW\otimes\bbW_*$.
 Let $\bM\in\mathbb T_{\fg}$ be a module such that all its simple constituents are of the form
  $\bV^{\boldsymbol{\lambda},{\boldsymbol{\mu}}}$ with $|\boldsymbol{\lambda}|+|{\boldsymbol{\mu}}|\leq n$. Then the length of $\bM^{\fs}$ in the category of $\mathfrak{sl}(n)$-modules equals the length of $\bM$ in
  $\mathbb T_{\fg}$.
\end{lem}
\begin{proof}
  It follows from (\ref{Schur}) and the fact that $\mathbb{S}_{\boldsymbol{\lambda}}(V_n)$ and $\mathbb{S}_{\boldsymbol{\mu}}(V_n^*)$ are nonzero  (since $\dim V_n\geq |\boldsymbol{\lambda}|,|\boldsymbol{\mu}|$) that
  $$(\mathbb{S}_{\boldsymbol{\lambda}}(\bV)\otimes \mathbb{S}_{\boldsymbol{\mu}}(\bV_*))^{\fs}=\mathbb{S}_{\boldsymbol{\lambda}}(V_n)\otimes \mathbb{S}_{\boldsymbol{\mu}}(V_n^*).$$
The description of the layers of the socle filtration of $\mathbb{S}_{\boldsymbol{\lambda}}(\bV)\otimes \mathbb{S}_{\boldsymbol{\mu}}(\bV_*)$ in (\ref{socle of S}) shows that
the length of $\mathbb{S}_{\boldsymbol{\lambda}}(\bV)\otimes \mathbb{S}_{\boldsymbol{\mu}}(\bV_*)$
  equals the length of $\mathbb{S}_{\boldsymbol{\lambda}}(V_n)\otimes \mathbb{S}_{\boldsymbol{\mu}}(V_n^*)$. Furthermore, since the socle $\bV^{\boldsymbol{\lambda},{\boldsymbol{\mu}}}$ of $\mathbb{S}_{\boldsymbol{\lambda}}(\bV)\otimes \mathbb{S}_{\boldsymbol{\mu}}(\bV_*)$ coincides with the set of vectors annihilated by all contraction maps (see (\ref{contractions})), and the set of vectors in $\mathbb{S}_{\boldsymbol{\lambda}}(V_n)\otimes \mathbb{S}_{\boldsymbol{\mu}}(V_n^*)$ annihilated by all contraction maps is the simple $\mathfrak{sl}(n)$-module  $V_n^{\boldsymbol{\lambda},{\boldsymbol{\mu}}}$, we obtain $(\bV^{\boldsymbol{\lambda},{\boldsymbol{\mu}}})^{\fs}=V_n^{\boldsymbol{\lambda},{\boldsymbol{\mu}}}$.  It then follows from left exactness that the functor
  $(\cdot)^\fs$  does not increase the length.

  Let $\bM\in\mathbb T_{\fg}$, and let $k(\bM)$ be the maximum of $|\boldsymbol{\lambda}|+|{\boldsymbol{\mu}}|$ over all simple constituents $\bV^{\boldsymbol{\lambda},{\boldsymbol{\mu}}}$ of $\bM$. We proceed by proving
  the statement by induction on $k(\bM)$ with the obvious base case $k(\bM)=0$. Consider an exact sequence
  $$0\to\bM\to \bI\to\bN \to 0,$$
  where $\bI$ is an injective hull of $\bM$ in $\mathbb T_{\fg}$. From the description of the socle filtration of an injective module in $\mathbb T_{\fg}$ (see (\ref{socle of S})), we have $k(\bN)<k(\bM)$. Therefore, the length $l(\bN)$ of $\bN$ equals the length $l(\bN^{\fs})$ of $\bN^{\fs}$ by the induction assumption. On the other hand, since $\bI$ is injective and hence isomorphic to a direct sum of $\mathbb{S}_{\boldsymbol{\lambda}}(\bV)\otimes \mathbb{S}_{\boldsymbol{\mu}}(\bV_*)$ with $|\boldsymbol{\lambda}|+|{\boldsymbol{\mu}}|\leq n$, the length of $\bI$ equals the length of $\bI^{\fs}$.  Now if $l(\bM^{\fs})<l(\bM)$, then
  $$l(\bN^{\fs})\geq l(\bI^{\fs})-l(\bM^{\fs})>l(\bI)-l(\bM)=l(\bN),$$
  which is a contradiction.
\end{proof}
\begin{cor}\label{cor:gen}  Let $\fs$ be a subalgebra of $\fg$ as in Lemma~\ref{aux0}, and let $\bM\in\widetilde{\mathbb T}_{\fg}$ be a module such that all its
  simple constituents are of the form
  $\bV^{\boldsymbol{\lambda},{\boldsymbol{\mu}}}$ with $|\boldsymbol{\lambda}|+|{\boldsymbol{\mu}}|\leq n$. Then $\bM=U(\fg)\bM^{\fs}$.
\end{cor}
\begin{proof} Since $\bM$ is a direct limit of modules of finite length it suffices to prove the statement for $\bM\in\mathbb{T}_{\fg}$. This can be easily done
  by induction on the length of $\bM$. Indeed, consider an exact sequence
  $0\to \bN\to \bM\to \bL\to 0$ with simple $\bL$. Lemma~\ref{aux0} implies that $0\to \bN^{\fs}\to \bM^{\fs}\to \bL^{\fs}\to 0$ is also exact, because the functor $(\cdot)^\fs$ is left exact and $l(\bL^{\fs})=l(\bM^{\fs})-l(\bN^{\fs})$. Now if $U(\fg)\bM^{\fs}\neq \bM$ then, since $U(\fg)\bN^{\fs}=\bN$ by the induction assumption, we obtain
  $U(\fg)\bM^{\fs}=\bN$. This implies $\bM^{\fs}=\bN^{\fs}$, and hence $l(\bL^s)=0$, which contradicts Lemma~\ref{aux0}.
  \end{proof}

  \begin{lem}\label{aux1} For any $\bM\in\mathbb T_{\fg,\fk}$ we have $U(\fg)\mathrm{R}(\bM)=\bM$.
  \end{lem}
  \begin{proof} Recall the definition of $k(\bM)$ from the proof of Lemma \ref{aux0}, and recall the decomposition (\ref{decomp V}).  Let $\bU$ be a subspace of $\bV$, and  $\bU_*$ be a subspace of $\bV_*$ such that  $\bV_r\subset\bU$  and  $(\bV_r)_*\subset\bU_*$, each of codimension $k(\bM)$.

Denote by $\fl\subset\fg$ the commutator subalgebra of  $\bU\otimes\bU_*$, and by $\operatorname{Res}_{\fl}$ the restriction functor from $\mathbb T_{\fg,\fk}$ to $\widetilde{\mathbb T}_{\fl}$. The identity (\ref{Schur}) implies that
    $k(\operatorname{Res}_{\fl}\bM)=k(\bM)$. By Corollary \ref{cor:gen} with $\fg=\fl$ and $\fs=\fk_r$, we get $\bM=U(\fl)\mathrm{R}(\bM)$. The statement follows.
\end{proof}

\begin{lem}\label{lem: R2} The functor $\mathrm{R}:\mathbb T_{\fg,\fk}\to \mathbb T_{\fg',\fk'}$ is exact and sends a simple module
  $\bV^{\boldsymbol{\lambda},{\boldsymbol{\mu}}}\in\mathbb T_{\fg,\fk}$ to the corresponding simple module $\bV^{\boldsymbol{\lambda},{\boldsymbol{\mu}}}\in\mathbb T_{\fg',\fk'}$, and hence induces an isomorphism between the
Grothendieck groups of $\mathbb T_{\fg,\fk}$ and $\mathbb T_{\fg',\fk'}$.
\end{lem}

\begin{proof}
Since $\bV^{\boldsymbol{\lambda},{\boldsymbol{\mu}}}$ is in fact an object of $\mathbb T_{\fg}$, the statement about simple modules follows by the argument concerning contraction maps from the proof of Lemma \ref{aux0}.

Since $\mathrm{R}$ is left exact, we have  the inequality
  \begin{equation}\label{eqn:length}
    l(\mathrm{R}(\bM))\leq l(\bM).
    \end{equation}
Thus, to prove exactness of $\mathrm{R}$ it suffices to show  that $\mathrm{R}$ preserves the length, i.e. $l(\bM)=l(\mathrm{R}(\bM))$. We prove this by induction on $l(\bM)$.
Consider an exact sequence of $\fg$-modules
$$0\to \bN\to \bM\to \bL\to 0,$$
such that $\bL$ is simple. By the induction hypothesis we have $l(\mathrm{R}(\bN))=l(\bN)$. If we assume that $l(\mathrm{R}(\bM))<l(\bM)$, then $l(\mathrm{R}(\bM))=l(\bN)$ and so $\mathrm{R}(\bN)=\mathrm{R}(\bM)$.  But then by Lemma~\ref{aux1}, we have $\bN=\bM$, which is a contradiction.
\end{proof}

\begin{cor}\label{I finte length} For any $\boldsymbol{\lambda},\boldsymbol{\mu}$, the module
$\Gamma_{\fg,\fk}((\mathbb{S}_{\boldsymbol{\lambda}}(\bV)\otimes \mathbb{S}_{{\boldsymbol{\mu}}}(\bV_*))^*)$ has finite length.
Hence, the module $\bI^{\boldsymbol{\lambda},{\boldsymbol{\mu}}}:=\Gamma_{\fg,\fk}((\bV^{{\boldsymbol{\mu}},\boldsymbol{\lambda}} )^*)$ has finite length and is an object of the category $\mathbb T_{\fg,\fk}$.
\end{cor}

\begin{proof}
It was proven in \cite{DPS} that $\Gamma_{\fg,\fg}((\mathbb{S}_{\boldsymbol{\lambda}}(\bV)\otimes \mathbb{S}_{{\boldsymbol{\mu}}}(\bV_*))^*)$  has finite length in  $\widetilde{\mathbb T}_{\fg}$ (see the proof of Proposition 4.5 in \cite{DPS} and note that the functor $\Gamma_{\fg,\fg}$ is denoted by $\mathcal B$ in \cite{DPS}).  Using  (\ref{formula RI}), the first claim follows by induction on the number $r$ of components in the decomposition of $\bV$. For the second claim, observe that Lemma~\ref{lem:dual} implies $\bI^{\boldsymbol{\lambda},{\boldsymbol{\mu}}}$ is isomorphic to a direct summand of the module $\Gamma_{\fg,\fk}((\mathbb{S}_{\boldsymbol{\mu}}(\bV)\otimes \mathbb{S}_{{\boldsymbol{\lambda}}}(\bV_*))^*)$.
\end{proof}

\begin{lem}\label{cor:Roninjectives} Let $\bI^{\boldsymbol{\lambda},{\boldsymbol{\mu}}}$ denote an injective hull of the simple module $\bV^{\boldsymbol{\lambda},{\boldsymbol{\mu}}}$ in $\mathbb T_{\fg,\fk}$,  and let $\bJ^{\boldsymbol{\lambda},{\boldsymbol{\mu}}}$ denote an injective hull of $\mathrm{R}(\bV^{\boldsymbol{\lambda},{\boldsymbol{\mu}}})$ in $\mathbb T_{\fg',\fk'}$. Then
  $$\mathrm{R}(\bI^{\boldsymbol{\lambda},{\boldsymbol{\mu}}})=\bigoplus_{\boldsymbol{\lambda'},\boldsymbol{\mu'},{\boldsymbol{\gamma}}} N^{\boldsymbol{\lambda}}_{\boldsymbol{\lambda}',{\boldsymbol{\gamma}}}N^{{\boldsymbol{\mu}}}_{{\boldsymbol{\mu}}',{\boldsymbol{\gamma}}}\bJ^{\boldsymbol{\lambda}',{\boldsymbol{\mu}}'}.$$
  \end{lem}

\begin{proof}  We have  $\bI^{\boldsymbol{\lambda},{\boldsymbol{\mu}}}\cong\Gamma_{\fg,\fk}((\bV^{{\boldsymbol{\mu}},\boldsymbol{\lambda}})^*)$ and $\bJ^{\boldsymbol{\lambda},{\boldsymbol{\mu}}}\cong\Gamma_{\fg',\fk'}((\bV^{{\boldsymbol{\mu}},\boldsymbol{\lambda}})^*)$.
Let $$\bP^{\boldsymbol{\lambda},{\boldsymbol{\mu}}}=\Gamma_{\fg,\fk}((\mathbb{S}_{\boldsymbol{\mu}}(\bV)\otimes \mathbb{S}_{\boldsymbol{\lambda}}(\bV_*))^*),\hspace{.5cm} \bQ^{\boldsymbol{\lambda},{\boldsymbol{\mu}}}=\Gamma_{\fg',\fk'}((\mathbb{S}_{\boldsymbol{\mu}}(\mathrm{R}(\bV))\otimes \mathbb{S}_{\boldsymbol{\lambda}}(\mathrm{R}(\bV_*))^*).$$ Then we have
\begin{equation}\label{eq P}
\bP^{\boldsymbol{\lambda},{\boldsymbol{\mu}}}\cong\bigoplus_{\boldsymbol{\lambda'},\boldsymbol{\mu'},{\boldsymbol{\gamma}}} N_{\boldsymbol{\lambda}',{\boldsymbol{\gamma}}}^{\boldsymbol{\lambda}} N_{{\boldsymbol{\mu}}',{\boldsymbol{\gamma}}}^{{\boldsymbol{\mu}}} \bI^{\boldsymbol{\lambda}',{\boldsymbol{\mu}}'}, \hspace{.5cm}
\bQ^{\boldsymbol{\lambda},{\boldsymbol{\mu}}}\cong\bigoplus_{\boldsymbol{\lambda'},\boldsymbol{\mu'},{\boldsymbol{\gamma}}} N_{\boldsymbol{\lambda}',{\boldsymbol{\gamma}}}^{\boldsymbol{\lambda}} N_{{\boldsymbol{\mu}}',{\boldsymbol{\gamma}}}^{{\boldsymbol{\mu}}} \bJ^{\boldsymbol{\lambda}',{\boldsymbol{\mu}}'}.
\end{equation}
Indeed, using Lemma~\ref{lem:dual}, we can deduce from (\ref{socle of S})  that
$$
(\mathbb{S}_{\boldsymbol{\mu}}(\bV)\otimes \mathbb{S}_{\boldsymbol{\lambda}}(\bV_*))^*=\bigoplus_{\boldsymbol{\lambda'},\boldsymbol{\mu'},{\boldsymbol{\gamma}}} N_{\boldsymbol{\lambda}',{\boldsymbol{\gamma}}}^{\boldsymbol{\lambda}} N_{{\boldsymbol{\mu}}',{\boldsymbol{\gamma}}}^{{\boldsymbol{\mu}}}(\bV^{\boldsymbol{\lambda}',{\boldsymbol{\mu}}'})^*,
$$
and then by applying $\Gamma_{\fg,\fk}$ to both sides we obtain (\ref{eq P}).

By (\ref{formula RI}), we have
$$
\mathrm{R}(\bP^{\boldsymbol{\lambda},{\boldsymbol{\mu}}})=\bigoplus_{\boldsymbol{\lambda'},\boldsymbol{\mu'},{\boldsymbol{\gamma}}}  N_{\boldsymbol{\lambda}',{\boldsymbol{\gamma}}}^{\boldsymbol{\lambda}} N_{{\boldsymbol{\mu}}',{\boldsymbol{\gamma}}}^{{\boldsymbol{\mu}}}  \bQ^{\boldsymbol{\lambda}',{\boldsymbol{\mu}}'}.
$$

Let $\mathfrak{I}_{\fg,\fk}$ denote the complexified Grothendieck group of the additive subcategory of $\mathbb T_{\fg,\fk}$  generated by indecomposable injective modules. Then $\{[\bI^{\boldsymbol{\lambda},{\boldsymbol{\mu}}}]\}$ and  $\{[\bP^{\boldsymbol{\lambda},{\boldsymbol{\mu}}}]\}$ both form a basis for  $\mathfrak{I}_{\fg,\fk}$. Let $A=(A_{\boldsymbol{\lambda}',{\boldsymbol{\mu}}'}^{\boldsymbol{\lambda},{\boldsymbol{\mu}}})$ be the change of basis matrix on $\mathfrak{I}_{\fg,\fk}$ given by (\ref{eq P}) which expresses $\bP^{\boldsymbol{\lambda},{\boldsymbol{\mu}}}$ in terms of $\bI^{\boldsymbol{\lambda},{\boldsymbol{\mu}}}$. The same matrix $A$ expresses  $\bQ^{\boldsymbol{\lambda},{\boldsymbol{\mu}}}$ in terms of $\bJ^{\boldsymbol{\lambda},{\boldsymbol{\mu}}}$ by (\ref{eq P}).

The functor $\mathrm{R}$ induces a linear operator from $\mathfrak{I}_{\fg,\fk}$ to $\mathfrak{I}_{\fg',\fk'}$ which is represented by the matrix $A$ with respect to both bases $\{[\bP^{\boldsymbol{\lambda},{\boldsymbol{\mu}}}]\}$ and $\{[\bQ^{\boldsymbol{\lambda},{\boldsymbol{\mu}}}]\}$. Hence, the  matrix which represents $\mathrm{R}$ with respect to the bases  $\{[\bI^{\boldsymbol{\lambda},{\boldsymbol{\mu}}}]\}$ and $\{[\bJ^{\boldsymbol{\lambda},{\boldsymbol{\mu}}}]\}$ is again $A$ as $A=AA(A^{-1})$.
\end{proof}

  \begin{cor}\label{cor:JHmultiplicities} The Jordan-H\"older multiplicities of the indecomposable injective modules $\bI^{\boldsymbol{\lambda},{\boldsymbol{\mu}}}$ are given by
    $$[\bI^{\boldsymbol{\lambda},{\boldsymbol{\mu}}}:\bV^{\boldsymbol{\lambda}',{\boldsymbol{\mu}}'}]=\sum_{\boldsymbol{\lambda'},\boldsymbol{\mu'},{\boldsymbol{\gamma}}_1,\dots,{\boldsymbol{\gamma}}_r}N^{\boldsymbol{\lambda}}_{{\boldsymbol{\gamma}}_1,\dots,{\boldsymbol{\gamma}}_r,\boldsymbol{\lambda}'}N^{{\boldsymbol{\mu}}}_{{\boldsymbol{\gamma}}_1,\dots,{\boldsymbol{\gamma}}_r,{\boldsymbol{\mu}}'}.$$

    \end{cor}

\begin{proof}
After applying the functor $\mathrm{R}$ to the module $\bI^{\boldsymbol{\lambda},{\boldsymbol{\mu}}}$ $(r-1)$ times, we obtain a direct sum of injective modules in the  category $\mathbb T_{\fg}$. The multiplicity of each indecomposable injective in this sum is thus determined by applying the matrix $A^{r-1}$ to $[\bI^{\boldsymbol{\lambda},{\boldsymbol{\mu}}}]$. The  Jordan-H\"older multiplicities of an indecomposable injective module in $\mathbb T_{\fg}$ are also given by the matrix $A$ (see \ref{socle of S}). Therefore,
$$
[\bI^{\boldsymbol{\lambda},{\boldsymbol{\mu}}}]=\sum (A^r)_{\boldsymbol{\lambda}',{\boldsymbol{\mu}}'}^{\boldsymbol{\lambda},{\boldsymbol{\mu}}}[\bV^{\boldsymbol{\lambda}',{\boldsymbol{\mu}}'}].
$$
\end{proof}


\subsection{The socle filtration of indecomposable injective objects in $\mathbb T_{\fg,\fk}$}

In this section, we describe the socle filtration of the injective objects $\bI^{\boldsymbol{\lambda},{\boldsymbol{\mu}}}$ in  $\mathbb T_{\fg,\fk}$.

We consider the restriction functor $$\Res_{\fk}:\mathbb T_{\fg,\fk} \rightarrow \mathbb T_{\fk},$$
where $\mathbb T_{\fk}$ denotes the category of integrable $\fk$-modules of finite length which satisfy the large annihilator condition for each $\fk_i$ (recall (\ref{decomp V})).
  Note that simple objects of $\mathbb T_{\fk}$ are outer tensor products of simple objects of the categories $\mathbb T_{\fk_i}$ for each $\fk_i$, $i=1,\ldots,r$, (recall that $\fk_i\cong\mathfrak{sl}(\infty)$); we will use the notation
    $$\bV^{\boldsymbol{\lambda}_1,\dots,\boldsymbol{\lambda}_r,{\boldsymbol{\mu}}_1,\dots,{\boldsymbol{\mu}}_r}:=\bV_1^{\boldsymbol{\lambda}_1,{\boldsymbol{\mu}}_1}\boxtimes\cdots\boxtimes\bV_r^{\boldsymbol{\lambda}_r,{\boldsymbol{\mu}}_r}.$$
    Injective hulls of simple objects in $\mathbb T_{\fk}$ will be denoted by $\bI_{\fk}^{\boldsymbol{\lambda}_1,\dots,\boldsymbol{\lambda}_r,{\boldsymbol{\mu}}_1,\dots,{\boldsymbol{\mu}}_r}$, and they are also outer tensor products of injective $\fk_i$-modules:
    $$\bI_{\fk}^{\boldsymbol{\lambda}_1,\dots,\boldsymbol{\lambda}_r,{\boldsymbol{\mu}}_1,\dots,{\boldsymbol{\mu}}_r}:=\left(\mathbb{S}_{\boldsymbol{\lambda}_1}(\bV_1)\otimes \mathbb{S}_{{\boldsymbol{\mu}}_1}(\bV_1)_*\right)\boxtimes\cdots\boxtimes \left(\mathbb{S}_{\boldsymbol{\lambda}_r}(\bV_r)
    \otimes \mathbb{S}_{{\boldsymbol{\mu}}_r}(\bV_r)_*\right).$$

    Recall that for every object $\bM$ in $ \mathbb T_{\fg,\fk}$ we denote by $k(\bM)$ the maximum of $|\boldsymbol{\lambda}|+|{\boldsymbol{\mu}}|$ for all simple constituents
    $\bV^{\boldsymbol{\lambda},{\boldsymbol{\mu}}}$ of $\bM$. Similarly for every object $\bX$ in $ \mathbb T_{\fk}$ we denote by $c(\bX)$ the maximum of
    $|\boldsymbol{\lambda}_1|+\dots+|\boldsymbol{\lambda}_r|+|{\boldsymbol{\mu}}_1|+\dots+|{\boldsymbol{\mu}}_r|$ for all simple constituents $\bV^{\boldsymbol{\lambda}_1,\dots,\boldsymbol{\lambda}_r,{\boldsymbol{\mu}}_1,\dots,{\boldsymbol{\mu}}_r}$ of $\bX$.
    It follows from Corollary \ref{cor:JHmultiplicities} that
\begin{equation}\label{socidentity}
  k(\bM)=k(\soc \bM),\quad c(\bX)=c(\soc \bX).
\end{equation}
The  identities
\begin{equation}\label{tensidentity}
  k(\bM\otimes \bN)=k(\bM)+k(\bN),\quad c(\bX\otimes \bY)=c(\bX)+c(\bY).
\end{equation}
 follow easily from the Littlewood--Richardson rule, and we leave their proof to the reader.

   \begin{lem}\label{lem:restriction} The restriction functor $\Res_{\fk}$ maps the category $ \mathbb T_{\fg,\fk}$ to the category $\mathbb T_{\fk}$, and it
      maps $\mathbb{S}_{\boldsymbol{\lambda}}(\bV)\otimes \mathbb{S}_{{\boldsymbol{\mu}}}(\bV_*)$ to an injective module. Furthermore, we have the identity
      $$c(\Res_{\fk}\bM)=k(\bM).$$
    \end{lem}
    \begin{proof} After applying identity (\ref{Schur}) $r$-times to $\mathbb{S}_{\boldsymbol{\lambda}}(\bV)\otimes \mathbb{S}_{{\boldsymbol{\mu}}}(\bV_*)$,   we get
      $$\Res_{\fk}(\mathbb{S}_{\boldsymbol{\lambda}}(\bV)\otimes \mathbb{S}_{{\boldsymbol{\mu}}}(\bV_*))\cong\bigoplus N^{\boldsymbol{\lambda}}_{\boldsymbol{\lambda}_1,\dots,\boldsymbol{\lambda}_r}N^{{\boldsymbol{\mu}}}_{{\boldsymbol{\mu}}_1,\dots,{\boldsymbol{\mu}}_r}\bI_{\fk}^{\boldsymbol{\lambda}_1,\dots,\boldsymbol{\lambda}_r,{\boldsymbol{\mu}}_1,\dots,{\boldsymbol{\mu}}_r}.$$
      This implies the first and the second assertions of the lemma. Identity  (\ref{socidentity}) implies that it
      is sufficient to prove the last assertion for $\bM=\mathbb{S}_{\boldsymbol{\lambda}}(\bV)\otimes \mathbb{S}_{{\boldsymbol{\mu}}}(\bV_*)$. Hence, this assertion follows from the above computation.
    \end{proof}

  \begin{conjecture}\label{lem:ext} Suppose $\operatorname{Ext}^k_{\mathbb T_{\fg,\fk}}(\bV^{\boldsymbol{\lambda}',{\boldsymbol{\mu}}'},\bV^{\boldsymbol{\lambda},{\boldsymbol{\mu}}})\neq 0$. Then $|\boldsymbol{\lambda}|-|\boldsymbol{\lambda}'|=|{\boldsymbol{\mu}}|-|{\boldsymbol{\mu}}'|=k$.
    \end{conjecture}

\begin{rem} For $\fk=\fg$, this was proven in \cite{DPS}. Proving this conjecture would imply that the category $ \mathbb T_{\fg,\fk}$ is Koszul. We prove the case $k=1$. \end{rem}

    \begin{prop}\label{lem:ext} Suppose $\operatorname{Ext}^1_{\mathbb T_{\fg,\fk}}(\bV^{\boldsymbol{\lambda}',{\boldsymbol{\mu}}'},\bV^{\boldsymbol{\lambda},{\boldsymbol{\mu}}})\neq 0$. Then $|\boldsymbol{\lambda}|-|\boldsymbol{\lambda}'|=|{\boldsymbol{\mu}}|-|{\boldsymbol{\mu}}'|=1$.
    \end{prop}
    \begin{proof}  Since $\bV^{\boldsymbol{\lambda}',{\boldsymbol{\mu}}'}$ is isomorphic to a simple constituent of $\bI^{\boldsymbol{\lambda},{\boldsymbol{\mu}}}$, we know by Corollary \ref{cor:JHmultiplicities} that
      $|\boldsymbol{\lambda}|-|\boldsymbol{\lambda}'|=|{\boldsymbol{\mu}}|-|{\boldsymbol{\mu}}'|=s\geq 1$. It remains to show that $s=1$. We will do this in two steps.

      First, we show that
      $\operatorname{Ext}^1_{\mathbb T_{\fg,\fk}}(\bV^{\boldsymbol{\lambda}',{\boldsymbol{\mu}}'},\mathbb{S}_{\boldsymbol{\lambda}}(\bV)\otimes \mathbb{S}_{{\boldsymbol{\mu}}}(\bV_*))\neq 0$ implies $s=1$. Consider
      a nonsplit short exact sequence in $\mathbb T_{\fg,\fk}$
      \begin{equation}\label{sequence Ext} 0\to \mathbb{S}_{\boldsymbol{\lambda}}(\bV)\otimes \mathbb{S}_{{\boldsymbol{\mu}}}(\bV_*)\to \bM\to \bV^{\boldsymbol{\lambda}',{\boldsymbol{\mu}}'}\to 0.\end{equation}
Let $\varphi:\bV^{\boldsymbol{\lambda}',{\boldsymbol{\mu}}'}\otimes\fg\to \mathbb{S}_{\boldsymbol{\lambda}}(\bV)\otimes \mathbb{S}_{{\boldsymbol{\mu}}}(\bV_*)$ be a cocyle which defines this extension.   By Lemma \ref{lem:restriction}, the module $\Res_{\fk}( \mathbb{S}_{\boldsymbol{\lambda}}(\bV)\otimes \mathbb{S}_{{\boldsymbol{\mu}}}(\bV_*))$ is injective in $\mathbb T_{\fk}$, and therefore the sequence (\ref{sequence Ext}) splits over $\fk$. Without loss of generality we may assume that $\varphi(\bV^{\boldsymbol{\lambda}',{\boldsymbol{\mu}}'}\otimes\fk)=0$. Then the cocycle condition implies that
      $\varphi:\bV^{\boldsymbol{\lambda}',{\boldsymbol{\mu}}'}\otimes(\fg/\fk)\to \mathbb{S}_{\boldsymbol{\lambda}}(\bV)\otimes \mathbb{S}_{{\boldsymbol{\mu}}}(\bV_*)$ is a nonzero homomorphism of $\fk$-modules.  Consequently, the image of $\varphi$
      contains a simple submodule in the socle of $\Res_{\fk}( \mathbb{S}_{\boldsymbol{\lambda}}(\bV)\otimes \mathbb{S}_{{\boldsymbol{\mu}}}(\bV_*))$. By Lemma \ref{lem:restriction}, we have
     $$\soc\Res_{\fk}(\mathbb{S}_{\boldsymbol{\lambda}}(\bV)\otimes \mathbb{S}_{{\boldsymbol{\mu}}}(\bV_*))=\bigoplus N^{\boldsymbol{\lambda}}_{\boldsymbol{\lambda}_1,\dots,\boldsymbol{\lambda}_r}N^{{\boldsymbol{\mu}}}_{{\boldsymbol{\mu}}_1,\dots,{\boldsymbol{\mu}}_r}V^{\boldsymbol{\lambda}_1,\dots,\boldsymbol{\lambda}_r,{\boldsymbol{\mu}}_1,\dots,{\boldsymbol{\mu}}_r}.$$

     In particular,
     $$c(V^{\boldsymbol{\lambda}_1,\dots,\boldsymbol{\lambda}_r,{\boldsymbol{\mu}}_1,\dots,{\boldsymbol{\mu}}_r})=|\boldsymbol{\lambda}_1|+\dots+|\boldsymbol{\lambda}_r|+|{\boldsymbol{\mu}}_1|+\dots+|{\boldsymbol{\mu}}_r|=|\boldsymbol{\lambda}|+|{\boldsymbol{\mu}}|$$
 for every simple submodule $V^{\boldsymbol{\lambda}_1,\dots,\boldsymbol{\lambda}_r,{\boldsymbol{\mu}}_1,\dots,{\boldsymbol{\mu}}_r}$ of $\soc\Res_{\fk}(\mathbb{S}_{\boldsymbol{\lambda}}(\bV)\otimes \mathbb{S}_{{\boldsymbol{\mu}}}(\bV_*))$.
Therefore, $$c(\bV^{\boldsymbol{\lambda}',{\boldsymbol{\mu}}'}\otimes (\fg/\fk))\geq |\boldsymbol{\lambda}|+|{\boldsymbol{\mu}}|,$$ and so (\ref{tensidentity}) implies
      $$c(\bV^{\boldsymbol{\lambda}',{\boldsymbol{\mu}}'})+c(\fg/\fk)\geq |\boldsymbol{\lambda}|+|{\boldsymbol{\mu}}|.$$
      Since $\fg/\fk\cong\bigoplus_{i\neq j}( \bV_i\otimes(\bV_j)_*)$,
 we have
      $$c(\bV^{\boldsymbol{\lambda}',{\boldsymbol{\mu}}'})=|\boldsymbol{\lambda}'|+|{\boldsymbol{\mu}}'|,\quad c(\fg/\fk)=2,$$
 and thus $|\boldsymbol{\lambda}|-|\boldsymbol{\lambda}'|+|\boldsymbol{\mu}|-|\boldsymbol{\mu}'|= 2s \leq 2$.
This yields $s=1$.

    Assume now to the contrary that $s\geq 2$. Set $$\bX=(\mathbb{S}_{\boldsymbol{\lambda}}(\bV)\otimes \mathbb{S}_{{\boldsymbol{\mu}}}(\bV_*))/\bV^{\boldsymbol{\lambda},{\boldsymbol{\mu}}}$$ and
      consider the long exact sequence of $\Ext$
      $$\dots\to\Hom_{\fg}(\bV^{\boldsymbol{\lambda}',{\boldsymbol{\mu}}'},\bX)\to \Ext^{1}_{\mathbb T_{\fg,\fk}}(\bV^{\boldsymbol{\lambda}',{\boldsymbol{\mu}}'},\bV^{\boldsymbol{\lambda},{\boldsymbol{\mu}}})\to\Ext^{1}_{\mathbb T_{\fg,\fk}}(\bV^{\boldsymbol{\lambda}',{\boldsymbol{\mu}}'}, \mathbb{S}_{\boldsymbol{\lambda}}(\bV)\otimes \mathbb{S}_{{\boldsymbol{\mu}}}(\bV_*))\to\dots.$$
Since $s\geq 2$, $\bV^{\boldsymbol{\lambda}',{\boldsymbol{\mu}}'}$ is not isomorphic to a submodule of $\soc \bX$, so   $\Hom_{\fg}(\bV^{\boldsymbol{\lambda}',{\boldsymbol{\mu}}'},\bX)=0$,
and by the already considered case when $s=1$, we have
     $$\Ext^{1}_{\mathbb T_{\fg,\fk}}(\bV^{\boldsymbol{\lambda}',{\boldsymbol{\mu}}'}, \mathbb{S}_{\boldsymbol{\lambda}}(\bV)\otimes \mathbb{S}_{{\boldsymbol{\mu}}}(\bV_*))=0.$$
      Hence, $\operatorname{Ext}^1_{\mathbb T_{\fg,\fk}}(\bV^{\boldsymbol{\lambda}',{\boldsymbol{\mu}}'},\bV^{\boldsymbol{\lambda},{\boldsymbol{\mu}}})= 0$, which is a contradiction.
    \end{proof}
    \begin{cor}\label{cor:soc} Suppose that $\bM\in{\mathbb T_{\fg,\fk}} $ has a simple socle $\bV^{\boldsymbol{\lambda},{\boldsymbol{\mu}}}$ and the multiplicity of $\bV^{\boldsymbol{\lambda}',{\boldsymbol{\mu}}'}$ in
      $\overline{\soc}^k \bM$ is nonzero. Then $|\boldsymbol{\lambda}|-|\boldsymbol{\lambda}'|=|{\boldsymbol{\mu}}|-|{\boldsymbol{\mu}}'|=k$.
    \end{cor}
\begin{proof}
This follows by induction on $|\boldsymbol{\lambda}|+|{\boldsymbol{\mu}}|$. By Proposition~\ref{lem:ext}, the module $\bM/\soc\bM$ embeds into a direct sum of injective indecomposable modules $\bigoplus \bI^{{\boldsymbol{\gamma}},{\boldsymbol{\nu}}}$ with simple socles  $\bV^{{\boldsymbol{\gamma}},{\boldsymbol{\nu}}}$ satisfying  $|\boldsymbol{\lambda}|-|{\boldsymbol{\gamma}}|=|{\boldsymbol{\mu}}|-|{\boldsymbol{\nu}}|=1$, and by induction each $\bI^{{\boldsymbol{\gamma}},{\boldsymbol{\nu}}}$ satisfies our claim. If the multiplicity of  $\bV^{\boldsymbol{\lambda}',{\boldsymbol{\mu}}'}$ is nonzero in  $\overline{\soc}^k \bM=\overline{\soc}^{k-1}( \bM/\soc\bM)\subset \overline{\soc}^{k-1}(\bigoplus \bI^{{\boldsymbol{\gamma}},{\boldsymbol{\nu}}})$, then  $|{\boldsymbol{\gamma}}|-|\boldsymbol{\lambda}'|=|{\boldsymbol{\nu}}|-|{\boldsymbol{\mu}}'|=k-1$. The result follows.
\end{proof}

Finally, by combining Corollary~\ref{cor:JHmultiplicities} and Corollary~\ref{cor:soc} we obtain the following.

    \begin{thm}\label{thm:socinj} The layers of the socle filtration of an indecomposable injective $\bI^{\boldsymbol{\lambda},{\boldsymbol{\mu}}}$ in $\mathbb T_{\fg,\fk}$ satisfy
      $$\overline{\soc}^k \bI^{\boldsymbol{\lambda},{\boldsymbol{\mu}}}\cong
\bigoplus_{\boldsymbol{\lambda'},\boldsymbol{\mu'}}
\bigoplus_{|{\boldsymbol{\gamma}}_1|+\dots+|{\boldsymbol{\gamma}}_r|=k}
      N^{\boldsymbol{\lambda}}_{{\boldsymbol{\gamma}}_1,\dots,{\boldsymbol{\gamma}}_r,\boldsymbol{\lambda}'}N^{{\boldsymbol{\mu}}}_{{\boldsymbol{\gamma}}_1,\dots,{\boldsymbol{\gamma}}_r,{\boldsymbol{\mu}}'}\bV^{\boldsymbol{\lambda}',{\boldsymbol{\mu}}'},$$
where $r$ is the number of (infinite) blocks in $\fk$ (see (\ref{decomp V})).
      \end{thm}

\begin{example}\label{example added}
Consider an injective hull of the adjoint representation of $\mathfrak{sl}(\infty)$ in the category $\mathbb T_{\fg,\fk}$ in the case that $\fk$ has $k$ (infinite) blocks. Then ${\boldsymbol{\lambda}}$ and ${\boldsymbol{\mu}}$ each consist of one box, and $\soc V^{{\boldsymbol{\lambda}},{\boldsymbol{\mu}}}=\mathfrak{sl}(\infty)$ and $\overline{\soc}^1 V^{{\boldsymbol{\lambda}},{\boldsymbol{\mu}}}=\mathbb C^k$, the trivial representation of dimension $k$.
The self-similarity effect mentioned in the introduction amounts here to the increase of the dimension of $\overline{\soc}^1$ by $1$ when the number of blocks of $\fk$ increases by $1$.
\end{example}

\begin{rem}
Let's observe that the category $\mathbb T_{\fg,\fk}$ is another example of an ordered tensor category as defined in \cite{CP1}. Indeed, the set $I$ in the notation of \cite{CP1} can be chosen as the set of pairs of Young diagrams $({{\boldsymbol{\lambda}},{\boldsymbol{\mu}}})$, and then the object $X_i$ for $i=({{\boldsymbol{\lambda}},{\boldsymbol{\mu}}})$ equals  $\bI^{\boldsymbol{\lambda},{\boldsymbol{\mu}}}$.
\end{rem}


\section{$\mathfrak{sl}(\infty)$-modules arising from category $\mathcal O$ for $\mathfrak{gl}(m|n)$}\label{super}

For the remainder of this paper, we let $\fk=\fk_1\oplus\fk_2$ be the commutator subalgebra of the Lie algebra preserving a fixed decomposition $\bV=\bV_1\oplus\bV_2$ such that both $\fk_1$ and $\fk_2$ are isomorphic to $\mathfrak{sl}(\infty)$ ($r=2$ in (\ref{decomp V})).

\subsection{Category $\mathcal{O}$ for the Lie superalgebra $\mathfrak{gl}(m|n)$}\label{cat O}

Let $\mathcal{O}_{m|n}$ denote the category of $\mathbb{Z}_{2}$-graded
modules over $\mathfrak{gl}(m|n)$ which when restricted to $\mathfrak{gl}(m|n)_{\bar{0}}$,
belong to the BGG category $\mathcal{O}_{\mathfrak{gl}(m|n)_{\bar{0}}}$  \cite[Section 8.2.3]{M}.
This category depends only on a choice of simple roots for the Lie
algebra $\mathfrak{gl}(m|n)_{\bar{0}}$, and not for all of $\mathfrak{gl}(m|n)$. We denote by $\mathcal{O}_{m|n}^{\mathbb{Z}}$ the Serre subcategory of $\mathcal{O}_{m|n}$
consisting of modules with integral weights. Any simple object in $\mathcal{O}_{m|n}^{\mathbb{Z}}$ is isomorphic
to $L\left(\lambda\right)$ (the unique simple quotient of the Verma
module $M\left(\lambda\right)$) for some $\lambda\in \Phi$, where $\Phi$ denotes the set of integral weights.
Any object in the category $\mathcal{O}_{m|n}^{\mathbb{Z}}$
has finite length.

We denote by $\mathcal{F}_{m|n}^{\mathbb{Z}}$ the Serre subcategory of $\mathcal{O}_{m|n}^{\mathbb{Z}}$
consisting of finite-dimensional modules. Let $\Pi:\mathcal{O}_{m|n}^{\mathbb{Z}}\rightarrow\mathcal{O}_{m|n}^{\mathbb{Z}}$ be the parity reversing functor.
We define the {\em reduced Grothendieck group} $K_{m|n}$ (respectively, $J_{m|n}$) to be the quotient of the Grothendieck group of  $\mathcal{O}_{m|n}^{\mathbb{Z}}$  (respectively, $\mathcal{F}_{m|n}^{\mathbb{Z}}$) by the relation $[\Pi M]=-[M]$.
The elements $\left[L\left(\lambda\right)\right]$ with $\lambda\in \Phi$ (respectively, $\lambda\in \Phi^+$) form a basis for $K_{m|n}$ (respectively, $J_{m|n}$).

We introduce an action of $\mathfrak{sl}(\infty)$ on $\bK_{m|n}:= K_{m|n}\otimes_{\mathbb Z}\mathbb C$ following Brundan \cite{B}.
Our starting point is to define the translation functors $\mathrm{E}_i$ and $\mathrm{F}_i$ on the category $\mathcal{O}_{m|n}^{\mathbb{Z}}$.
Consider the invariant form $\operatorname{str}(XY)$ on $\mathfrak{gl}\left(m|n\right)$ and let $X_j$, $Y_j$ be a pair of $\mathbb Z_2$-homogeneous dual bases of $\mathfrak{gl}(m|n)$ with respect to this form.
Then for two $\mathfrak{gl}\left(m|n\right)$-modules $V$ and $W$ we define the operator $$\Omega:V\otimes W\to V\otimes W,$$
$$\Omega(v\otimes w):=\sum_j(-1)^{p(X_j)(p(v)+1)}X_j v\otimes Y_j w,$$
where $p(X_j)$ denotes the parity of the $\mathbb Z_2$-homogeneous element $X_j$.
It is easy to check that $\Omega\in\operatorname{End}_{\mathfrak{gl}\left(m|n\right)}(V\otimes W)$.
Let $U$ and $U^*$ denote the natural and conatural $\mathfrak{gl}\left(m|n\right)$-modules. For every $M\in\mathcal{O}_{m|n}^{\mathbb{Z}}$ we let $\mathrm{E}_i(M)$
(respectively, $\mathrm{F}_i(M)$) be the
generalized eigenspace of $\Omega$ in $M\otimes U^*$ (respectively, $M\otimes U$) with eigenvalue $i$.
Then, as it follows from \cite{BLW}, the functor $\cdot\otimes U^*$ (respectively, $\cdot\otimes U$) decomposes into the direct sum of functors $\oplus_{i\in\mathbb Z}\mathrm{E}_i(\cdot)$ (respectively,
$\oplus_{i\in\mathbb Z}\mathrm{F}_i(\cdot)$). Moreover, the functors $\mathrm{E}_i$ and $\mathrm{F}_i$ are mutually adjoint functors on $\mathcal{O}_{m|n}^{\mathbb{Z}}$.
 We will denote by $e_i$ and $f_i$ the linear operators which the functors $\mathrm{E}_i$ and
$\mathrm{F}_i$ induce on $\bK_{m|n}$.

If we identify $e_i$ and $f_i$ with the Chevalley generators $E_{i,i+1}$ and $F_{i+1,i}$ of  $\mathfrak{sl}(\infty)$, then $\bK_{m|n}$ inherits the natural structure of a
$\mathfrak{sl}(\infty)$-module. This follows from \cite{B,BLW}. Another proof can be obtained by using Theorem 3.11 of \cite{CS} and (\ref{K embeds}) below.  Weight spaces with respect to the diagonal subalgebra $\mathfrak h \subset \mathfrak{sl}(\infty)$ correspond to the complexified reduced Grothendieck groups of the blocks of $\mathcal{O}_{m|n}^{\mathbb{Z}}$.

Let  $\bJ_{m|n}:= J_{m|n}\otimes_{\mathbb Z}\mathbb C$, and let $\bT_{m|n}\subset \bK_{m|n}$ denote the subspace generated by the classes $[M(\lambda)]$ of all Verma modules $M(\lambda)$ for $\lambda\in\Phi$. Let furthermore $\bW_{m|n}\subset\bJ_{m|n}$ denote the subspace generated by the classes $[K(\lambda)]$ of all Kac modules $K(\lambda)$ for $\lambda\in\Phi^+$  (for the definition of a Kac module see for example \cite{B}).
Then $\bT_{m|n}$ is an  $\mathfrak{sl}\left(\infty\right)$-submodule isomorphic to $\bV^{\otimes m}\otimes \bV^{\otimes n}_*$ and $\bW_{m|n}$ is a submodule of $\bT_{m|n}$ isomorphic to $\Lambda^{m}\bV\otimes\Lambda^{n}\bV_*$ \cite{B}. To see this, let $\{v_i\}_{i\in\mathbb Z}$
and $\{w_i\}_{i\in\mathbb Z}$ be the standard dual bases in $\bV$ and $\bV_*$ (i.e. $\mathfrak h$-eigenbases in $\bV$ and $\bV_*$), and let $\bar\lambda:=\lambda+(m-1,\dots,1,0|0,-1,\dots,1-n)$,
$$m_\lambda:=v_{\bar\lambda_1}\otimes \cdots\otimes v_{\bar\lambda_m}\otimes v^*_{-\bar\lambda_{m+1}}\otimes\cdots\otimes v^*_{-\bar\lambda_{m+n}}.$$
The map $[M(\lambda)]\mapsto m_{\lambda}$ establishes an isomorphism $\bT_{m|n}\cong \bV^{\otimes m}\otimes \bV^{\otimes n}_*$, and
restricts to an isomorphism \begin{align*}\bW_{m|n}&\cong \Lambda^{m}\bV\otimes\Lambda^{n}\bV_{*}\\
 [K(\lambda)]&\mapsto k_{\lambda}:=v_{\bar\lambda_1}\wedge \cdots\wedge v_{\bar\lambda_m}\otimes v^*_{-\bar\lambda_{m+1}}\wedge\cdots\wedge v^*_{-\bar\lambda_{m+n}}.\end{align*}

\begin{lem}\label{lem-annihilator} The $\mathfrak{sl}(\infty)$-module $\bK_{m|n}$  satisfies the large annihilator condition as a module over $\fk_1$ and $\fk_2$, that is,  $\Gamma_{\fg,\fk}(\bK_{m|n})=\bK_{m|n}$.
\end{lem}
\begin{proof}
Note that an $\mathfrak{sl}(\infty)$-module $\bM$ satisfies the large annihilator condition over $\fk_1$ and $\fk_2$ if and only if for each $x\in\bM$, we have $e_i x=f_i x=0$ for all but finitely many $i\in\mathbb{Z}$.
Indeed, if  $e_i x=f_i x=0$ for all but finitely many $i\in\mathbb{Z}$, then the subalgebra generated by the $e_i,f_i$ that annihilate $x$ contains the commutator subalgebra of the centralizer of a finite-dimensional subalgebra. The other direction is also clear.

Since the classes of simple $\mathfrak{gl}(m|n)$-modules $[L(\lambda)]$ form a basis of $\bK_{m|n}$, we just need to show that  for each $L(\lambda)$ we have $\mathrm{E}_i(L(\lambda))=\mathrm{F}_i(L(\lambda))= 0$
for almost all $i\in\mathbb Z$. However, since $\bT_{m|n}$ satisfies the large annihilator condition, we know that the analogous statement is true for $M(\lambda)$. Therefore, since $L(\lambda)$ is a quotient of $M(\lambda)$, the exactness of the functors $\mathrm{E}_i$ and $\mathrm{F}_i$ implies the desired statement for $L(\lambda)$.
\end{proof}

If we consider the Cartan involution $\sigma$ of $\mathfrak{sl}(\infty)$, $\sigma (e_i)=-f_i$, $\sigma(f_i)=-e_i$, we obtain
\begin{equation}\label{cartan}
\langle g x,y\rangle=-\langle x,\sigma(g)y\rangle
\end{equation}
 for all $g\in\mathfrak{sl}(\infty)$.
If $\bX$ is a $\mathfrak{sl}(\infty)$-module, we denote by $\bX^\vee$ the twist of the algebraic dual $\bX^*$ by $\sigma$.
Note that $ (\bV^{{\boldsymbol{\lambda}},\boldsymbol{\mu}})^\vee= \bV^{\boldsymbol{\mu},{\boldsymbol{\lambda}}}$. Hence, if  $\bX$ is a semisimple object of finite length in $\widetilde{Tens}_{\mathfrak{g}}$, then $\bX^\vee$ is an injective hull of $\bX$ in $\widetilde{Tens}_{\mathfrak{g}}$.

Let $\mathcal P_{m|n}$ denote the semisimple subcategory of $\mathcal{O}_{m|n}^{\mathbb{Z}}$ which consists of projective $\mathfrak {gl}(m|n)$-modules, and let $P_{m|n}$ denote the reduced Grothendieck group of $\mathcal P_{m|n}$.
The $\mathfrak{sl}(\infty)$-module $\bP_{m,n}:= P_{m|n}\otimes_{\mathbb Z}\mathbb C$ is the socle of $\bT_{m,n}$ \cite[Theorem 3.11]{CS}.
Note that for any  projective module $P\in\mathcal P_{m|n}$ the functor $\operatorname{Hom}_{\mathfrak {gl}(m|n)}(P,\cdot)$ on  $\mathcal{O}_{m|n}^{\mathbb{Z}}$ is exact, and  for any module $M\in\mathcal F_{m|n}$  the functor  $\operatorname{Hom}_{\mathfrak {gl}(m|n)}(\cdot,M)$ on $\mathcal P_{m|n}$ is exact.
Moreover, we have the dual bases in $\bK_{m|n}$ and $\bP_{m|n}$ given by the classes of irreducible modules and indecomposable projective modules, respectively.

Consider the pairing $\bK_{m|n}\times \bP_{m|n}\to \mathbb C$ defined by
$$\langle [M],[P]\rangle:=\operatorname{dim}\operatorname{Hom}_{\mathfrak{gl}(m|n)}(P,M).$$
Since the functors $\mathrm{E}_i$ and $\mathrm{F}_i$ are adjoint, we have
$$\langle e_i x,y\rangle=\langle x, f_iy\rangle$$
and
$$\langle f_i x,y\rangle=\langle x, e_iy\rangle,$$
for all $i\in\mathbb{Z}$, $x\in \bK_{m|n}$, $y\in\bP_{m|n}$.
Thus, there is an embedding of $\mathfrak{sl}(\infty)$-modules
\begin{equation}\label{K embeds}\Psi: \bK_{m|n}\hookrightarrow \bP^\vee_{m|n}\end{equation}
 given by  $[M]\mapsto\langle [M],\cdot\rangle$.

\begin{thm}\label{thm: T injective} The $\mathfrak{sl}(\infty)$-module $\bK_{m|n}$ is an injective hull in the category $\mathbb T_{\fg,\fk}$ of the semisimple module $\bP_{m|n}$. Furthermore, there is an isomorphism
$$
\bK_{m|n}\cong\bigoplus_{|\boldsymbol{\lambda}|=m,|\boldsymbol{\mu}|=n}\bI^{\boldsymbol{\lambda},\boldsymbol{\mu}}\otimes(Y_{\boldsymbol{\lambda}}\otimes Y_{{\boldsymbol{\mu}}})
$$
where $Y_{\boldsymbol{\lambda}}$, $Y_{\boldsymbol{\mu}}$ are irreducible modules over $S_m$ and $S_n$ respectively, and $\bI^{\boldsymbol{\lambda},\boldsymbol{\mu}}$ is an injective hull of the simple module $\bV^{\boldsymbol{\lambda},\boldsymbol{\mu}}$ in $\mathbb T_{\fg,\fk}$. Consequently, the layers of the socle filtration of $\bK_{m|n}$  are given by
 $$
\overline{\soc}^k  \bK_{m|n}\cong\bigoplus_{|\boldsymbol{\lambda}|=m,|\boldsymbol{\mu}|=n}(\overline{\soc}^k \bI^{\boldsymbol{\lambda},\boldsymbol{\mu}})^{\oplus (\dim Y_{\boldsymbol{\lambda}} \dim Y_{{\boldsymbol{\mu}}})}
$$
where
$$\overline{\soc}^k \bI^{\boldsymbol{\lambda},{\boldsymbol{\mu}}}\cong
\bigoplus_{\boldsymbol{\lambda'},\boldsymbol{\mu'}}
\bigoplus_{|{\boldsymbol{\gamma}}_1|+|{\boldsymbol{\gamma}}_2|=k}
      N^{\boldsymbol{\lambda}}_{{\boldsymbol{\gamma}}_1,{\boldsymbol{\gamma}}_2,\boldsymbol{\lambda}'}N^{{\boldsymbol{\mu}}}_{{\boldsymbol{\gamma}}_1,{\boldsymbol{\gamma}}_2,{\boldsymbol{\mu}}'}\bV^{\boldsymbol{\lambda}',{\boldsymbol{\mu}}'}.$$
\end{thm}
\begin{proof}

The module  $\Gamma_{\fg,\fk}(\bP_{m|n}^\vee)$  is an injective hull of the semisimple module $\bP_{m|n}$ in the category $\mathbb T_{\fg,\fk}$, so it suffices to show that the  image of  $\bK_{m|n}$ under the embedding (\ref{K embeds}) equals $\Gamma_{\fg,\fk}(\bP_{m|n}^\vee)$.
The fact that $\Psi(\bK_{m|n})\subset \Gamma_{\fg,\fk}(\bP_{m|n}^\vee)$ follows from Lemma~\ref{lem-annihilator}. Herein, we will identify $\bK_{m|n}$ with its image $\Psi(\bK_{m|n})=\mathrm{span}\{\langle l_{\lambda},\cdot\rangle\mid \lambda\in\Phi\}$, where $ l_\lambda:=[L(\lambda)]$.

Now $\soc(\Gamma_{\fg,\fk}(\bP_{m|n}^\vee))=\bP_{m|n}$, since $\bP_{m|n}$ is semisimple, and  $\soc\bT_{m|n}=\bP_{m|n}$ by \cite[Theorem 3.11]{CS}.
Therefore, since $\bT_{m|n}\subset\bK_{m|n}\subset\Gamma_{\fg,\fk}(\bP_{m|n}^\vee)$, we have $\soc\bK_{m|n}=\bP_{m|n}$.

We will show that  $\bK_{m|n}=\Gamma_{\fg,\fk}(\bP_{m|n}^\vee)$. To accomplish this, we  use the existence of the dual bases $p_\lambda:=[P(\lambda)]\in\bP_{m|n}$ and $l_\lambda\in\bK_{m|n}$, where
 $L(\lambda)$ denotes the irreducible $\mathfrak{gl}(m|n)$-module with highest weight $\lambda\in\Phi$ and $P(\lambda)$ is a projective cover of  $L(\lambda)$.

Fix $\omega\in \Gamma_{\fg,\fk}(\bP^{\vee}_{m|n})$. To prove that $\omega\in\bK_{m|n}=\mathrm{span}\{\langle l_{\lambda},\cdot\rangle\mid \lambda\in\Phi\}$, it suffices to show that $\omega(p_\lambda)=0$ for almost all $\lambda\in\Phi$.
For each $q,r\in\mathbb{Z}$, with $q<r$, we let $\fg_{q,r}:=\fg_{q}^-\oplus\fg_{r}^+$, where $\fg_{q}^-$ is the subalgebra of $\fg$ generated by $e_i,f_i$ for $i<q$ and $\fg_{r}^+$ is the subalgebra of $\fg$ generated by $e_i,f_i$ for $i>r$.
By the annihilator condition,  $\omega$ is $\mathfrak{g}_{q,r}$-invariant for suitable $q$ and $r$. Fix such $q$ and $r$. Then since   $\omega$ is $\mathfrak{g}_{q,r}$-invariant, it suffices to show that $p_\lambda\in\mathfrak{g}_{q,r}\bP_{m|n}$ for almost all $\lambda\in\Phi$ .

If  $p_\lambda\in\bP_{m|n}\cap(\mathfrak{g}_{q,r}\bT_{m|n})$, then  $p_\lambda\in\mathfrak{g}_{q,r}\bP_{m|n}$. Indeed, for any $\fg_{q,r}$-module $\bM$ we have
$$
\fg_{q,r}\bM=\bigcap_{\varphi\in\mathrm{Hom}_{\fg_{q,r}}(\bM,\mathbb{C})} \ker\ \varphi.
$$
Now any $\fg_{q,r}$-module homomorphism $\varphi:\bP_{m|n}\rightarrow\mathbb{C}$ lifts to a $\fg_{q,r}$-module  homomorphism $\varphi:\bK_{m|n}\rightarrow\mathbb{C}$, since the trivial module $\mathbb{C}$ is injective in the full subcategory of $\fg_{q,r}$-mod consisting of integrable finite-length $\fg_{q,r}$-modules satisfying the large annihilator condition \cite{DPS}. Hence, the claim follows.

For each $\lambda\in\Phi$ we define $\mathrm{supp}(\bar{\lambda})$ to be the multiset $\{\bar\lambda_1,\ldots,\bar\lambda_m,-\bar\lambda_{m+1},\ldots,-\bar\lambda_{m+n}\}$, where
$$\bar\lambda:=\lambda+(m-1,\dots,1,0|0,-1,\dots,1-n).$$
The set of $\lambda\in\Phi$ such that
 $\mathrm{supp}(\bar{\lambda})\cap(\mathbb{Z}_ {<(q-m-n)}\cup\mathbb{Z}_ {>(r+m+n)})=\emptyset$ is finite.
Hence, to finish the proof of the theorem, it suffices to show the following.

\begin{lem}\label{lemma for thm} If $\mathrm{supp}(\bar{\lambda})\cap\mathbb{Z}_ {<(q-m-n)}\neq\emptyset$, then $p_\lambda\in\mathfrak{g}_{q}^-\bT_{m|n}$. Similarly, if
$\mathrm{supp}(\bar{\lambda})\cap\mathbb{Z}_ {>(r+m+n)}\neq\emptyset$, then $p_\lambda\in\mathfrak{g}_{r}^+\bT_{m|n}$.
\end{lem}
\begin{proof}  We will prove the first statement; the proof of the second statement is similar.
  We can write $p_\lambda=\sum_{\nu}c_{\nu}m_{\nu}$, where each $c_\nu\in\mathbb{Z}_{>0}$ and $m_{\nu}=[M(\nu)]$ is the class of the Verma module $M(\nu)$ over $\mathfrak{gl}(m|n)$ of highest weight $\nu\in\Phi$.

We claim that  $\mathrm{supp}(\bar{\nu})\cap\mathbb{Z}_ {<q}\neq\emptyset$ for every $m_\nu$ which occurs in the decomposition of   $p_\lambda$.
Indeed, recall that $P(\lambda)$ is a direct
    summand in the induced module $\operatorname{Ind}^{\mathfrak{gl}(m|n)}_{\mathfrak{gl}(m|n)_{\bar 0}}P^0(\lambda)$, where $P^0(\lambda)$ is a projective cover of the simple
    $\mathfrak{gl}(m|n)_{\bar 0}$-module with highest weight $\lambda$. Now
\begin{equation}\label{KL}
  [P^0(\lambda)]=\sum_{w\in \mathcal W}b_{w\cdot \lambda}[M^0({w\cdot\lambda})],
  \end{equation}
    where $M^0({\mu})$ denotes the Verma module over $\mathfrak{gl}(m|n)_{\bar 0}$ with highest weight $\mu$, $\mathcal W$ denotes the Weyl group of  $\mathfrak{gl}(m|n)_{\bar 0}$ and $w\cdot\lambda$ denotes the $\rho_{\bar{0}}$-shifted action of $\mathcal W$. The isomorphism of $\mathfrak{gl}(m|n)$-modules
    $$M({\mu})\cong \operatorname{Ind}^{\mathfrak{gl}(m|n)}_{\mathfrak{gl}(m|n)_{\bar 0}\oplus \mathfrak{gl}(m|n)_1}M^0(\mu)$$
    implies that
    $$\operatorname{Ind}^{\mathfrak{gl}(m|n)}_{\mathfrak{gl}(m|n)_{\bar 0}}M^0(\mu)\cong \operatorname{Ind}^{\mathfrak{gl}(m|n)}_{\mathfrak{gl}(m|n)_{\bar 0}\oplus \mathfrak{gl}(m|n)_1}(M^0(\mu)\otimes U(\mathfrak{gl}(m|n)_1).$$
    Therefore, $\operatorname{Ind}^{\mathfrak{gl}(m|n)}_{\mathfrak{gl}(m|n)_{\bar 0}}M^0(\mu)$ admits a filtration by Verma modules $M({\mu+\gamma})$ where $\gamma$ runs over the set of
    weights of $U(\mathfrak{gl}(m|n)_1)$. Since $\mathrm{supp}(\gamma)\subset\{-m-n,\dots,m+n\}$ for every $\gamma$, we have  $$|(\overline{\mu+\gamma})_i-\bar{\mu}_i|\leq m+n.$$ Combining this with (\ref{KL}) we obtain that for each
    $i\leq m+n$, $|\bar\nu_i-\bar\lambda_{w(i)}|<m+n$, for some $w\in\mathcal W$. The claim follows.

    Following the notations of Lemma \ref{lem:tech} from the appendix, we set $$\bbW_1=\mathrm{span}\{v_i,\,|\,i<q\},\quad \bbW_2=\mathrm{span}\{v_j,\,|\,j\geq q\}.$$
    Then $\fg^-_q=\mathfrak{sl}(\bbW_1)=\fs$. By above, every $m_{\nu}$ occurring in the decomposition of $p_\lambda$ is contained in $\bY_{m|n}$. Hence
    $p_\lambda\in\bY_{m|n}$. Since we also have $p_\lambda\in\soc \bT_{m|n}$, Lemma \ref{lem:tech} implies that $p_\lambda\in\mathfrak{g}_{q}^-\bT_{m|n}$.
\end{proof}
Hence, $\bK_{m|n}=\Gamma_{\fg,\fk}(\bP_{m|n}^\vee)$, and the description of the socle filtration now follows from Theorem~\ref{thm:socinj}.
\end{proof}

\subsection{The symmetric group action on $\bK_{m|n}$}
Recall that we have a natural action of the product of symmetric groups $S_m\times S_n$ on $\bT_{m|n}$, which commutes with the $\mathfrak{sl}(\infty)$-module structure on $\bT_{m|n}$. Moreover, it follows from
\cite[Sect. 6]{DPS} that
\begin{equation}\label{endsym}
  \operatorname{End}_{\mathfrak{sl}(\infty)}(\bT_{m|n})=\operatorname{End}_{\mathfrak{sl}(\infty)}(\bP_{m|n})=\mathbb C[S_m\times S_n].
\end{equation}

A similar result is true for $\bK_{m|n}$:
\begin{prop}\label{end}
  $$\operatorname{End}_{\mathfrak{sl}(\infty)}(\bK_{m|n})=\operatorname{End}_{\mathfrak{sl}(\infty)}(\bP_{m|n})=\mathbb C[S_m\times S_n].$$
\end{prop}
\begin{proof} Recall that $\bP_{m|n}$ is the socle of $\bK_{m|n}$ by Theorem~\ref{thm: T injective}.  Every $\varphi\in \operatorname{End}_{\mathfrak{sl}(\infty)}(\bK_{m|n})$ maps the socle to the socle, hence we have a homomorphism
  \begin{equation}\label{homom} \operatorname{End}_{\mathfrak{sl}(\infty)}(\bK_{m|n})\to\operatorname{End}_{\mathfrak{sl}(\infty)}(\bP_{m|n}).\end{equation}
Let $\bK'_{m|n}=\bK_{m|n}/\bP_{m|n}$. By Theorem \ref{thm:socinj}, for every simple module
  $\bV^{\lambda,\mu}$ we have
  $$[\bK'_{m|n}:\bV^{\lambda,\mu}][\bP_{m|n}:\bV^{\lambda,\mu}]=0.$$
  Therefore,  every $\varphi\in
  \operatorname{End}_{\mathfrak{sl}(\infty)}(\bK_{m|n})$ such that
  $\varphi(\bP_{m|n})=0$ is identically zero,  since for such $\varphi$ the socle of $\operatorname{im} \varphi$ is zero. In other words,
  homomorphism (\ref{homom}) is injective. The surjectivity follows from the fact that every
  $\varphi:\bP_{m|n}\to \bP_{m|n}\hookrightarrow\bK_{m|n}$ extends to $\tilde\varphi:\bK_{m|n}\to\bK_{m|n}$ by the injectivity of $\bK_{m|n}$.
  \end{proof}
  \subsection{The Zuckerman functor $\Gamma_{\mathfrak{gl}(m|n)}$ and
    the category $\mathcal F_{m|n}^{\mathbb Z}$}
  Let us recall the definition of the derived Zuckerman functor. A systematic treatment of the Zuckerman functor for Lie superalgebras can be found in \cite{S}.
  Assume that $M$ is a finitely generated $\mathfrak{gl}(m|n)$-module which is  semisimple over the
  Cartan subalgebra of $\mathfrak{gl}(m|n)$. Let $\Gamma_{\mathfrak{gl}(m|n)}(M)$ denote the subspace of $\mathfrak{gl}(m|n)_0$-finite vectors. Then
  $\Gamma_{\mathfrak{gl}(m|n)}(M)$ is a finite-dimensional $\mathfrak{gl}(m|n)$-module, and hence $\Gamma_{\mathfrak{gl}(m|n)}$ is a left exact functor from the category
  of finitely generated $\mathfrak{gl}(m|n)$-modules, semisimple over the Cartan subalgebra, to the category $\mathcal F_{m|n}$ of finite-dimensional modules. The corresponding right derived functor $\Gamma^i_{\mathfrak{gl}(m|n)}$ is called the \emph{$i$-th derived Zuckerman functor}. Note that $\Gamma^i_{\mathfrak{gl}(m|n)}(X)=0$ for
  $i>\dim{\mathfrak{gl}(m|n)_0}-(m+n)$. We are interested in the restriction of this functor
    $$\Gamma^i_{\mathfrak{gl}(m|n)}:\mathcal O_{m|n}^{\mathbb Z}\to \mathcal F_{m|n}^{\mathbb Z}.$$
    Let us consider the linear operator $\gamma:\bK_{m|n}\to\bJ_{m|n}$ given by
    $$\gamma([M])=\sum_i(-1)^i[\Gamma^i_{\mathfrak{gl}(m|n)} M].$$
    This operator is well defined as for any short exact sequence of $\mathfrak{gl}(m|n)$-modules
    $$0\to N\to M\to L\to 0,$$
    we have the Euler characteristic identity
    $$\gamma ([M])=\gamma ([N])+\gamma ([L]).$$
      It is well known that     $\Gamma^i_{\mathfrak{gl}(m|n)}$  commutes with the functors $\cdot\otimes U$ and $\cdot\otimes U^*$, and with the projection to the
      block  $(\mathcal O_{m|n}^{\mathbb Z})_{\chi}$ with a
      fixed central character $\chi$. Therefore, $\gamma$
    is a homomorphism of $\mathfrak{sl}(\infty)$-modules.
    \begin{prop}\label{Zuckerman} The homomorphism $\gamma$ is given by the formula
      \begin{equation}\label{gamma} \gamma=\sum_{s\in S_m\times S_n}\operatorname{sgn}(s) s,\end{equation}
      where the action of $s$ on $\bK_{m|n}$ is defined in Proposition \ref{end}.
    \end{prop}
    \begin{proof} By Proposition \ref{end},
      it suffices to check the equality (\ref{gamma}) on vectors in $\bT_{m|n}$, which amounts to checking that for all Verma modules $M(\lambda)$
\begin{equation}\label{verma}
  \gamma([M(\lambda)])=\sum_{s\in S_m\times S_n}\operatorname{sgn}(s) [M(s\cdot\lambda)],
  \end{equation}
      where $s\cdot \lambda=s(\lambda+\rho)-\rho$ and $\rho=(m-1,\dots,0|0,-1,\dots,1-n)$.

      Consider the functor $\operatorname{Res}_0$ of restriction to $\mathfrak{gl}(m|n)_0$. This is an exact functor from the category
      of finitely generated  $\mathfrak{gl}(m|n)$-modules, semisimple over the Cartan subalgebra, to the similar category of $\mathfrak{gl}(m|n)_0$-modules. It is
      clear from the definition of $\Gamma^i_{\mathfrak{gl}(m|n)}$ that
      \begin{equation}\label{res}
       \operatorname{Res}_0\Gamma^i_{\mathfrak{gl}(m|n)}=\Gamma^i_{\mathfrak{gl}(m|n)_0}\operatorname{Res}_0.
        \end{equation}
        Recall that every Verma module $M(\lambda)$ over $\mathfrak{gl}(m|n)$ has a  finite filtration with successive quotients isomorphic to Verma modules
        $M^0(\mu)$ over $\mathfrak{gl}(m|n)_0$. Hence by (\ref{res}) it
        suffices to check the analogue of (\ref{verma}) for even Verma modules:
        \begin{equation}\label{evenverma}
       \gamma^0([M^0(\lambda)])=\sum_{s\in S_m\times S_n}\operatorname{sgn}(s) [M^0(s\cdot\lambda)],
       \end{equation}
       where $\gamma^0$ is the obvious analogue of $\gamma$. To prove
       (\ref{evenverma}) we observe that
       $[M^0(\lambda)]=[M^0(\lambda)^\vee]$ where $X^\vee$ stands for
       the contragredient dual of $X$.

       It is easy to compute
       $\Gamma^i_{\mathfrak{gl}(m|n)_0}M^0(\lambda)^\vee$.
       Let $\mathfrak t$ denote the Cartan subalgebra of
       $\mathfrak{gl}(m|n)$, and let $\mathfrak n_0^{+}$, $\mathfrak n_0^{-}$ be the maximal
       nilpotent ideals of the Borel and opposite Borel subalgebras  of $\mathfrak{gl}(m|n)_0$, respectively.
       From the definition of the derived Zuckerman functor, the
       following holds
       for any $\mu\in \Phi^+$
       $$\operatorname{Hom}_{\mathfrak{gl}(m|n)_0}(L^0(\mu),\Gamma^i_{\mathfrak{gl}(m|n)_0} M)\simeq\operatorname{Ext}^i(L^0(\mu),M),$$
       where the extension is taken in the category of modules
       semisimple over $\mathfrak t$. If
       $M=M^0(\lambda)^\vee$, then $M$ is cofree over $U(\mathfrak n^+_0)$ and therefore
       $$\operatorname{Ext}^i(L^0(\mu),M^0(\lambda)^\vee)\simeq
       \operatorname{Hom}_{\mathfrak t}(H_i({\mathfrak n^-_0},L^0(\mu)), \mathbb C_\lambda).$$
       Now we apply Kostant's theorem to conclude that
       $$\Gamma^i_{\mathfrak{gl}(m|n)_0}M^0(\lambda)^\vee=\begin{cases}
       L^0(\mu)\ \  \text{if}\,\, \mu=s\cdot\lambda \,\,\text{for}\,\, s\in
       S_m\times S_n,\ l(s)=i,\\ 0\ \hspace{1cm} \text{otherwise}. \end{cases}$$
     Here $\mu$ is the only dominant weight in $(S_m\times S_n)\cdot\lambda$ and hence $s$ is unique. Moreover, if $\lambda+\rho$ is a singular weight then
 $\Gamma^i_{\mathfrak{gl}(m|n)_0}M^0(\lambda)^\vee=0$ for all $i$.
    Combining this with the Weyl character formula
    $$[L^0(\mu)]=\sum_{s\in S_m\times S_n}\operatorname{sgn}(s) [M^0(s\cdot\mu)]$$
      we obtain (\ref{evenverma}), and hence the proposition.
    \end{proof}
    \begin{cor}\label{fin-dim} We have
      $\bJ_{m|n}=\gamma(\bK_{m|n})$ and
      $\bK_{m|n}=\bJ_{m|n}\oplus\ker\gamma$.  In
      particular, $\bJ_{m|n}$ is an injective hull of $\bW_{m|n}\cong\Lambda^m\bV\otimes\Lambda^n\bV_*$.
      \end{cor}
Recall that  $\bW_{m|n}\subset\bJ_{m|n}$
denotes the subspace generated by the classes of all Kac modules.
Let $\mathcal Q_{m|n}$ denote the additive subcategory of $\mathcal
F_{m|n}^{\mathbb{Z}}$  which consists
of projective finite-dimensional $\mathfrak {gl}(m|n)$-modules, and
let $Q_{m|n}$ denote the reduced Grothendieck group of $\mathcal
Q_{m|n}$.    It was proven  in \cite[Theorem 3.11]{CS} that
$\bQ_{m|n}:= Q_{m|n}\otimes_{\mathbb{Z}}\mathbb{C}$ is the socle of
the module $\bW_{m|n}$, implying that $\bQ_{m|n}\cong
\bV^{(m)^\perp,(n)^{\perp}}$, where $\perp$ indicates the conjugate partition.
Corollary \ref{fin-dim} implies  the following.
 \begin{cor}\label{fin-dim2} $\bJ_{m|n}$ is an injective hull of $\bQ_{m|n}$,  and the socle filtration of $\bJ_{m|n}$ is
\[
\overline{\soc}^i \bJ_{m|n}\cong\left(\bV^{\left(m-i\right)^{\perp}\left(n-i\right)^{\perp}}\right)^{\oplus (i+1)}.
\]
      \end{cor}


\subsection{The Duflo--Serganova functor and the tensor filtration}

In this section, we discuss the relationship between the Duflo--Serganova functor and submodules of the $\mathfrak{sl}(\infty)$-modules $\bK_{m|n}$ and $\bJ_{m|n}$.

Let $\fa=\fa_{\bar{0}}\oplus\mathfrak{a}_{\bar{1}}$ be  a finite-dimensional contragredient Lie superalgebra.
For any odd element $x\in\fa_{\bar{1}}$ which satisfies $\left[x,x\right]=0$,  the \emph{Duflo\textendash Serganova functor $DS_{x}$}
is defined by
\begin{align*}
DS_{x}:\ \fa-\operatorname{mod} &\rightarrow\fa_x-\operatorname{mod}  \\ M &\mapsto\mathrm{ker}_{M}x/xM,
\end{align*}
where $\mathrm{ker}_{M}x/xM$ is a module over the Lie superalgebra $\fa_{x}:=\fa^{x}/[x,\fa]$ (here $\fa^{x} $ denotes the centralizer of $x$ in $\fa$) \cite{DS}.
In what follows we set
$$M_x:=DS_{x}(M).$$
The Duflo--Serganova functor $DS_{x}$ is a  symmetric monoidal functor,  \cite{DS}, see also Proposition 5 in \cite{Slec}.

It is known that the functor $DS$ is not exact, nevertheless it induces  a homomorphism $ds_x$ between the reduced Grothendieck groups of the categories
$\fa$-mod and $\fa_x$-mod defined by $ds_x([M])=[M_x]$. (Recall that "reduced" indicates passage to the quotient  by the relation $[\Pi M]=-[M]$, where $\Pi$ is the parity reversing functor.)
This follows from the following statement, see Section 1.1 in \cite{GS}.
\begin{lem}\label{reduced}
For every exact sequence of $\fa$-modules
$$0\to M_1\xrightarrow{\psi} M_2\xrightarrow{\varphi} M_3\to 0$$
there exists an exact sequence of $\fa_x$-modules
$$0 \to E\to DS_x(M_1)\xrightarrow{DS_x(\psi)} DS_x(M_2)\xrightarrow{DS_x(\varphi)} DS_x(M_3)\to\Pi E\to 0,$$
for an appropriate $\fa_x$-module $E$.
\end{lem}
\begin{proof}  Set $E:=\operatorname{Ker}(DS_x(\psi))$, $E':=\operatorname{Coker}(DS_x(\varphi))$, and consider the exact sequence
$$0\to E\to DS_x(M_1)\to DS_x(M_2)\to DS_x(M_3)\to E'\to 0.$$
The odd morphism $\psi^{-1}x\varphi^{-1}:DS_x(M_3)\to DS_x(M_1)$ induces an isomorphism $E'\to\Pi E$.
  \end{proof}

In \cite{HR} the existence of the homomorphism $ds_x$ was proven for finite-dimensional modules.
\begin{rem} If $0\to C_1\to\dots\to C_k\to 0$ is a complex of $\fa$-modules with odd differentials, the Euler characteristic of this complex is defined as the element  $\sum_{i=1}^k [C_i]$ in the reduced Grothendieck group. If $H_i$ denotes the $i$-th cohomology group, then
  $$\sum_{i=1}^k[C_i]=\sum_{i=1}^k[H_i].$$
  The absence of the usual sign follows from the relation $[\Pi M]=-[M]$ and the fact that the differentials are odd. For example, for an acyclic complex
  $0\to X\to\Pi X\to 0$ the Euler characteristic is zero.
  \end{rem}
Let $\fa=\mathfrak{gl}(m|n)$ and suppose $\operatorname{rank}x=k$.  Then $\fa_x\cong \mathfrak{gl}(m-k|n-k)$.
Let  $\mathcal{O}^{ind}_{m|n}$ be the category whose objects are direct limits of objects in $\mathcal{O}_{m|n}$. Then by Lemma 5.2 in \cite{CS} the restriction
of $DS_x$ to $\mathcal{O}_{m|n}$ is a well-defined functor
$$DS_{x}:\ \mathcal{O}_{m|n} \rightarrow\mathcal{O}^{ind}_{m-k|n-k}.$$

\begin{lem}\label{lem:commute} The functor $DS_{x}:\ \mathcal{O}^{\mathbb Z}_{m|n} \rightarrow(\mathcal{O}^{\mathbb Z}_{m-k|n-k})^{ind}$ commutes with translation functors.
\end{lem}
\begin{proof} Recall that $U$ is the natural  $\mathfrak{gl}\left(m|n\right)$-module. Since $DS$ is a monoidal functor, we have a canonical isomorphism
  $$(M\otimes U)_x\simeq M_x\otimes U_x.$$
  Moreover, a direct computation shows that $U_x$ is isomorphic to the natural $\mathfrak{gl}(m-k|n-k)$-module. We will use these observations to show that there is
  a canonical isomorphism
  \begin{equation}\label{goal1}
   \mathrm{E}_i(M_x)\simeq (\mathrm{E}_i(M))_x.
\end{equation}

   Recall the notations of Section 3.1. Define the homomorphism of $\mathfrak{gl}(m|n)$-modules
  $$\omega_{m|n}:\mathbb C\to\mathfrak{gl}(m|n)\otimes \mathfrak{gl}(m|n), \quad 1\mapsto\sum (-1)^{p(X_j)}X_j\otimes Y_j.$$
  We have $DS_x(\omega_{m|n})=\omega_{m-k|n-k}$. Consider the composition
  $$\Omega: M\otimes U\xrightarrow{1\otimes \omega_{m|n}\otimes 1}M\otimes \mathfrak{gl}(m|n)\otimes \mathfrak{gl}(m|n)\otimes U
  \xrightarrow{r_M\otimes l_U}M\otimes U,
  $$
  where $r_M :M\otimes \mathfrak{gl}(m|n)\to M$ is the morphism of right action, and  $l_U:\mathfrak{gl}(m|n)\otimes U\to U$ is the morphism of left action.
  The morphism $DS_x(\Omega):M_x\otimes U_x\to M_x\otimes U_x$ is defined in a similar manner in the category of $\mathfrak{gl}(m-k|n-k)$-modules.
  Recall that
  $$\mathrm{E}_i(M)=\{v\in M\otimes U\,|\,(\Omega-i)^Nv=0\quad\text{for some}\,\,N>0\};$$
  similarly
  $$\mathrm{E_i}(M_x)=\{v\in M_x\otimes U_x\,|\,(DS_x(\Omega)-i)^Nv=0\quad\text{for some}\,\,N>0\}.$$
  This implies the existence of the isomorphism (\ref{goal1}) as desired.

  The proof for $\mathrm{F}_i$ is similar.
  \end{proof}

We are going to strengthen the result of \cite{CS} by proving the following proposition.
\begin{prop}\label{DSO} The restriction of $DS_x$ to $\mathcal{O}_{m|n}$  is a well-defined functor
  $$DS_{x}:\ \mathcal{O}_{m|n} \rightarrow\mathcal{O}_{m-k|n-k}.$$
\end{prop}
To prove the proposition we first consider the case when $k=1$.
\begin{lem}\label{lem:rank1} If $k=1$, then  the restriction of $DS_x$ to  $\mathcal{O}_{m|n}$ is a well-defined functor
  $$DS_{x}:\ \mathcal{O}_{m|n} \rightarrow\mathcal{O}_{m-1|n-1}.$$
\end{lem}
\begin{proof} By Theorem 5.1 in \cite{CS} we may assume without loss of generality that $x$ is a generator of the root space $\mathfrak{gl}(m|n)_{\alpha}$ for some $\alpha=\pm(\varepsilon_i-\delta_j)$.
  Moreover, we can choose a Borel subalgebra $\mathfrak{b}\subset\mathfrak{gl}(m|n)$ so that $\alpha$ is a simple root. Let $M$ be an object in the category
  $\mathcal{O}_{m|n}$ and $M^\mu$ denote the weight space of weight $\mu$. The set of all weights of $M$ is denoted by $\operatorname{supp}M$.
  Let $x_{\mu}:M^\mu\to M^{\mu+\alpha}$ be the restriction of $x$ as an operator on $M$. Then
  $$M_x=\oplus_{\mu\in\operatorname{supp}M}M^{\mu}_x\quad\text{where}\quad M^\mu_x=\Ker x_\mu/x_{\mu-\alpha}(M^{\mu-\alpha}).$$
  Let us first check that all weight multiplicities of $M_x$ are finite with respect to the Cartan subalgebra
  $\fh_x:=\Ker\varepsilon_i\cap\Ker\delta_j$ of $\fg_x$. We have to show that for any
  $\nu\in \fh^*_x$
  \begin{equation}\label{goal}
   \sum_{\mu\in\operatorname{supp}M,\mu|_{\fh_x}=\nu}\dim  M^\mu_x<\infty.
    \end{equation}

 Note that $\dim M_{x}^\mu\neq 0$ implies $(\mu,\alpha)=0$, by $\mathfrak{sl}(1|1)$-representation theory.  If $(\mu',\alpha')=0$ and $\mu|_{\fh_x}=\mu'|_{\fh_x}$, then
  $\mu-\mu'\in\mathbb C\alpha$. Denote by $\Delta_s$ the set of simple roots of $\mathfrak{b}$.
  Since $M$ is an object of  $\mathcal{O}_{m|n}$, $M$ has a finite filtration by highest weight modules. Therefore it suffices to consider the case when $M$ is
  a highest weight module.
  Let $\lambda$ be the highest weight of $M$. Then every $\mu\in\operatorname{supp}M$
  has the form
  $\lambda-\sum_{\beta\in\Delta_s}k_\beta\beta$ for some $k_\beta\in\mathbb Z_{\geq 0}$ satisfying $k_\alpha\leq 1+\sum_{\beta\in\Delta_S\setminus\alpha}k_\beta$.
 Therefore, for any $\mu\in\operatorname{supp}M$ the set $(\mu+\mathbb C\alpha)\cap\operatorname{supp}M$ is finite.
 Hence, for any $\nu\in  \fh^*_x $ the set of $\mu\in\operatorname{supp}M$ such that
 $\mu|_{\fh_x}=\nu$ and $(\mu,\alpha)=0$ is finite. Since all weight spaces of $M$ are finite dimensional, this implies (\ref{goal}).

To finish the proof we observe that Lemma \ref{lem:commute} implies $\mathrm{E}_i(M_x)=\mathrm{F}_i(M_x)=0$ for almost all $i\in\mathbb Z$. Now for each $i\in\operatorname{supp}(\bar\lambda)$ , at least one of the $\mathrm{E}_i,\mathrm{E}_{i+1},\mathrm{F}_i,\mathrm{F}_{i+1}$
 does not annihilate $L_{\fg_x}(\lambda)$. Together this implies that the set $S_M$ of all weights $\lambda$  satisfying $[M_x:L_{\fg_x}(\lambda)]\neq 0$ is a finite set. On the other hand, since $M_x$ has finite weight multiplicities, every simple constituent occurs in $M_x$
 with finite multiplicity. Hence  $M_x$ has finite length.
\end{proof}
  \begin{proof} Now we prove Proposition \ref{DSO} by induction on $\operatorname{rank}(x)=k$. By Theorem 5.1 in \cite{CS},  $x$ is $B_0$-conjugate to $x_1+\dots+x_k$, where $x_i\in\mathfrak{gl}(m|n)_{\alpha_i}$ for
  some linearly independent set of mutually orthogonal odd roots $\beta_1,\dots,\beta_k$. So without loss of generality we may suppose that
  $x=x_1+\dots+x_k$. Let $y=x_1+\dots+x_{k-1}$. Choose  $h_y\in \fh_{x_k}$ and $h_{x_k}\in \fh_y$  such that  $\alpha(h_y),\alpha(h_{x_k})\in \mathbb Z$ for all roots $\alpha$ of $\mathfrak{gl}(m|n)$,
  $[h_y,y]=y$ and $[h_{x_k},x_k]=x_k$. Assume that $M\in\mathcal O_{m|n}$ and $\operatorname{supp}M\in\lambda +Q$, where $Q$ is the root lattice. Then $\operatorname{ad}h_y-\lambda(h_y)$ and $\operatorname{ad}h_{x_k}-\lambda(h_{x_k})$ define a
  $\mathbb Z\times\mathbb Z$-grading on $M$ and the differentials $y$ and $x_k$ form a bicomplex. Moreover, $M_x$ is nothing but the cohomology
  $\bigoplus_r H^r(y+x_k,M)$ of the total complex.

  Consider the second term
  $$E_2^{p,q}(M)=H^p(x_k,H^q(y,M))$$
  of the spectral sequence of this bicomplex.
  By the induction assumption $M_y\in \mathcal O_{m-k+1|n-k+1}$, and in particular, $H^q(y,M)\neq 0$ for finitely many $q$.
  The induction assumption implies that $H^p  (x_k,H^q(y,M))\in\mathcal O_{m-k|n-k}$ does not vanish
  for finitely many $p$. This yields $\bigoplus_{p,q} E_2^{p,q}(M)\in \mathcal O_{m-k|n-k}$. Since $\bigoplus_r H^r(y+x_k,M)$ is a subquotient of
  $\bigoplus_{p,q} E_2^{p,q}(M)$, we obtain
  $$M_x=\bigoplus_r H^r(y+x_k,M)\in\mathcal{O}_{m-k|n-k}.$$
\end{proof}
Next note that the restriction of $DS_x$ to $\mathcal O^{\mathbb Z}_{m|n}$ is a well-defined functor
  $$\mathcal O^{\mathbb Z}_{m|n}\to \mathcal O^{\mathbb Z}_{m-k|n-k}.$$

Since $DS_{x}$ is a well-defined functor from $\mathcal {O}_{m|n}^{\mathbb Z} $ to $\mathcal {O}_{m-k|n-k}^{\mathbb Z} $ we see that  $ds_x:K_{m|n}\rightarrow K_{m-k|n-k}$ is a well-defined group homomorphism.

  \begin{lem}\label{lem:fin-dim} If $x=x_1+\dots+x_k$ with commuting $x_1,\dots,x_k$ of rank $1$, then on $K_{m|n}$ we have the identity
  $$ds_x=ds_{x_k}\circ\dots\circ ds_{x_1}.$$
\end{lem}
\begin{proof} We retain the notation of the proof of Proposition \ref{DSO}. Clearly, it suffices to check that
$$ds_x=ds_{x_k}\circ ds_y,$$  where  $y=x_1+\dots+x_{k-1}$.
  The Euler characteristic of the $E_s$-terms of the spectral sequence from the proof of Proposition \ref{DSO} remains unchanged for $s\geq 2$:
  $$[\bigoplus_{p,q}E_2^{p,q}(M)]=[\bigoplus_{p,q}E_s^{p,q}(M)].$$
   As the spectral sequence converges to  $[M_x]$, we obtain
  $$ds_{x_k}\circ ds_y([M])=[\bigoplus_{p,q}E_2^{p,q}(M)]=[M_x]=ds_x([M]).$$
\end{proof}
For the category of finite-dimensional modules the above statement is proven in \cite{HR}.

\begin{prop} The complexification $ds_x:\bK_{m|n}\rightarrow\bK_{m-k|n-k}$ is a homomorphism of $\mathfrak{sl}(\infty)$-modules, as is its restriction  $ds_x:\bJ_{m|n}\rightarrow\bJ_{m-k|n-k}$ to the $\mathfrak{sl}(\infty)$-submodule $\bJ_{m|n}:= J_{m|n}\otimes_{\mathbb Z}\mathbb C$.
\end{prop}

\begin{proof} This follows from the fact that the Duflo--Serganova functor commutes with translation functors, see Lemma \ref{lem:commute}.
  \end{proof}

\begin{rem}
 Note that in \cite{HR} the ring $J_{m|n}$ is denoted by $\mathcal{J}_G$ where $G=GL(m|n)$.
\end{rem}

Let $X_{\fa}=\left\{ x\in\fa_{\bar{1}}:\left[x,x\right]=0\right\} $,
and let \begin{equation}
\mathcal{B}_{\fa}=\left\{ B\subset\Delta_{iso}\mid B=\left\{ \beta_{1},\ldots,\beta_{k}\mid\left(\beta_{i},\beta_{j}\right)=0,\ \beta_{i}\ne\pm\beta_{j}\right\} \right\} \label{eq:iso}
\end{equation} be the set
of subsets of linearly independent mutually orthogonal isotropic roots of $\fa$.
Then the orbits of the action of the adjoint group $G_{\bar{0}}$ of $\fa_{\bar{0}}$ on $X_{\fa}$ are in one-to-one
correspondence with the  orbits of the Weyl group $\mathcal W$ of $\fa_{\bar{0}}$ on $\mathcal{B}_{\fa}$
via the correspondence
\begin{equation}
B=\left\{ \beta_{1},...,\beta_{k}\right\} \mapsto x=x_{\beta_{1}}+\cdots+x_{\beta_{k}}\in X_{\fa},\label{eq:bijection for x and B}
\end{equation}
where each $x_{\beta_{i}}\in\fa_{\beta_{i}}$
is chosen to be nonzero \cite[Theorem 4.2]{DS}.

\begin{lem}
\label{prop:independent} Let $\fa=\mathfrak{gl}(m|n)$. Fix $x\in X_{\fa}$ and set $k=|B_x|$, where $B_x\in \mathcal{B}_{\fa}$ corresponds to $x$. The homomorphism $ds_{x}:{J}_{m|n}\rightarrow{J}_{m-k|n-k}$ depends only $k$, and not on $x$.
\end{lem}
\begin{proof}
This follows from the description of $ds_{x}$ given in \cite[Theorem 10]{HR}, using the fact that supercharacters of finite-dimensional modules are invariant under the Weyl group $\mathcal W=S_m\times S_n$ of $\mathfrak{gl}(m|n)$. If $B_{1}, B_{2}\in\mathcal{B}$ with
$|B_{1}|=|B_{2}|$ then there exists $w\in\mathcal W$ satisfying: $\pm\beta\in w\left(B_{1}\right)$
if and only if $\pm\beta\in B_{2}$. So if $f\in{J}_{m|n}$
we have that
\[ds_{x_1}(f)=
f\vert_{\beta_{1}^{1},\ldots,\beta_{k}^{1}=0}=w\left(f\right)\vert_{w\left(\beta_{1}^{1}\right),\ldots,w\left(\beta_{k}^{1}\right)=0}=w\left(f\right)\vert_{\beta_{1}^{2},\ldots,\beta_{k}^{2}=0}=f\vert_{\beta_{1}^{2},\ldots,\beta_{k}^{2}=0}
=ds_{x_2}(f).\]
\end{proof}

Note that Lemma \ref{prop:independent} does not hold if we replace ${J}_{m|n}$ with ${K}_{m|n}$.

\begin{rem}\label{ds k} Since the homomorphism  $ds_{x}:\bJ_{m|n}\rightarrow\bJ_{m-k|n-k}$  does not depend on $x$, we denote it by  $ds^k$, where $|B_x|=k$, and
  we let $ds:=ds^1$.
\end{rem}

Now we introduce a filtration of an $\mathfrak{sl}(\infty)$-module $\bM$, whose layers are tensor modules.

\begin{defn}

The {\em tensor filtration} of an $\mathfrak{sl}(\infty)$-module $\bM$ is defined inductively by $$\tens^0\bM:=\tens \bM:=\Gamma_{\fg,\fg}(\bM), \hspace{1cm} \tens^{i} \bM:=p_i^{-1}(\tens(\bM/(\tens^{i-1}\bM))),$$
where $p_i:\bM\to \bM/(\tens^{i-1} \bM)$ is the natural projection.

We also use the notation $\overline{\tens}^i \bM=\tens^i \bM /\tens^{i-1}\bM$.
\end{defn}
Note that $\tens\bM$ is the maximal tensor submodule of $\bM$.

\begin{example}

The socle of $\bJ_{1|1}$ is isomorphic to the adjoint module of $\mathfrak{sl}(\infty)$, and $\overline{\soc}^1 \bJ_{1|1}=\mathbb C\oplus\mathbb C$. Note that this is a special case of Example~\ref{example added} in the case that $\fk$ has two infinite blocks.

Consider now the tensor filtration of $\bJ_{1|1}$. This filtration also has  length $2$, $\tens \bJ_{1|1}=\bW_{1|1}\cong \bV\otimes \bV_{*}$
 and  $\overline{\tens}^1 \bJ_{1|1}\cong \mathbb C$.
The module  $\bJ_{1|1}$ admits a nice matrix realization.
Indeed, we can identify the $\mathfrak{sl}\left(\infty\right)$-module $\bW_{1|1}$ with the matrix realization of $\mathfrak{gl}(\infty)$  (see Section~\ref{prelims}), and then extend it
 by the diagonal matrix $D$ which has entries $D_{ii}=1$ for $i\geq1$ and $0$ elsewhere. The action of $\mathfrak{sl}\left(\infty\right)$ in this realization of $\bJ_{1|1}$ is the adjoint action.

\end{example}

\begin{prop}
\label{thm:module structure} For each $k$, let $ds^{k}:\bJ_{m|n}\to\bJ_{m-k|n-k}$ be the homomorphism induced by the Duflo--Serganova functor (see Remark~\ref{ds k}). Set $t:=1+\min\left\{ m,n\right\} $ and let $\bM_{k}^{t}:=\ker ds^{k}$. Consider the filtration
 of $\mathfrak{sl}\left(\infty\right)$-modules
$$\bM_{1}^{t}\subset\bM_{2}^{t}\subset\cdots\subset\bM_{t}^{t}=\bJ_{m|n}.$$ Then $\bM_{1}^{t}=\bW_{m|n}$ and $\bM_{k+1}^{t}/\bM_{k}^{t}\cong\Lambda^{m-k}\bV\otimes\Lambda^{n-k}\bV_{*}$. This filtration is the tensor filtration of $\bJ_{m|n}$, that is, $\tens^{k-1}\bJ_{m|n}=\ker ds^k$.
\end{prop}
\begin{proof} In the proof we let $m$ and $n$ vary.
It follows from \cite[Theorems 17 and 20]{HR} that for every $m,n\in\mathbb Z_{>0}$ the map $ds:{\bJ}_{m|n}\rightarrow{\bJ}_{m-1|n-1}$ is surjective and the kernel is spanned by the
classes of Kac modules. So we have an exact sequence of $\mathfrak{sl}(\infty)$-modules
\[
0\to\bW_{m|n}\to\bJ_{m|n}\stackrel{ds}{\to}\bJ_{m-1|n-1}\to0.
\]
Thus, we obtain the following diagram of
$\mathfrak{sl}\left(\infty\right)$-modules for each $l=|m-n|$, in which the horizontal arrows represent the map $ds$.
\[
\begin{array}{cccccccccc}
\twoheadrightarrow & \bM_{5}^{5} & \twoheadrightarrow & \bM_{4}^{4} & \twoheadrightarrow & \bM_{3}^{3} & \twoheadrightarrow & \bM_{2}^{2} & \twoheadrightarrow & \bM_{1}^{1}\\
 & \cup &  & \cup &  & \cup &  & \cup &  & \cup\\
\twoheadrightarrow & \bM_{4}^{5} & \twoheadrightarrow & \bM_{3}^{4} & \twoheadrightarrow & \bM_{2}^{3} & \twoheadrightarrow & \bM_{1}^{2} & \twoheadrightarrow & 0\\
 & \cup &  & \cup &  & \cup &  & \cup\\
\twoheadrightarrow & \bM_{3}^{5} & \twoheadrightarrow & \bM_{2}^{4} & \twoheadrightarrow & \bM_{1}^{3} & \twoheadrightarrow & 0\\
 & \cup &  & \cup &  & \cup\\
\twoheadrightarrow & \bM_{2}^{5} & \twoheadrightarrow & \bM_{1}^{4} & \twoheadrightarrow & 0\\
 & \cup &  & \cup\\
\twoheadrightarrow & \bM_{1}^{5} & \twoheadrightarrow & 0
\end{array}
\]
By induction we get  $\bM_{k+1}^{t}/\bM_{k}^{t}\cong\bM_{1}^{t-k}=\bW_{m-k|n-k}$, and by \cite{B}, $\bW_{m-k|n-k}\cong\Lambda^{m-k}\bV\otimes\Lambda^{n-k}\bV_{*}$. Hence, the first claim follows.

For the second claim, suppose for sake of contradiction that for some $k$, the module $\bM_{k+1}^{t}/\bM_{k}^{t}$
is not the maximal tensor submodule of $\bJ_{m|n}/\bM_{k}^{t}$.
By projecting to $\bJ_{m-k|n-k}$, we obtain that
$\bM_{1}^{t}$ is not the maximal tensor submodule of $\bJ_{m|n}$, for some $m,n$.
Since $\bM_{1}^{t}=\bW_{m|n}\cong\Lambda^{m}\bV\otimes\Lambda^{n}\bV_{*}$
is injective in the category $\mathbb{T}_{\mathfrak{g}}$ \cite{DPS}, this implies that $\soc\bJ_{m|n}$ is larger than $\soc \bM_{1}^{t}$, which is a contradiction since  $\soc\bJ_{m|n}=\soc\bW_{m|n}=\bP_{m|n}$.
\end{proof}

In the rest of this subsection, we fix $x$ to be a generator of the root space corresponding to $\delta_j-\varepsilon_i$. We denote by $ds_{ij}:\bK_{m|n}\to \bK_{m-1|n-1}$
the $\mathfrak{sl}(\infty)$-module homomorphism $ds_x$.
\begin{prop} We have
$$\bigcap_{i,j}\ker ds_{ij}=\bT_{m|n}.$$
\end{prop}
\begin{proof}
It follows from \cite{HR} that $ds_{ij}[M]=0$ if and only if $e^{\varepsilon_i}-e^{\delta_j}$ divides the supercharacter $\operatorname{sch} M$ of $M$.
Hence, $[M]$ lies in the intersection of kernels of all $ds_{ij}$ if and only if $\prod_{i,j} (e^{\varepsilon_i}-e^{\delta_j})$  divides $\operatorname{sch} M$. This means that
$\operatorname{sch} M$ is a linear combination of supercharacters induced from the parabolic subalgebra $\mathfrak{gl}(m|n)_{\bar 0}\oplus \mathfrak{gl}(m|n)_{1}$. Therefore,
$\operatorname{sch} M$ is a linear combination of supercharacters of Verma modules.
\end{proof}

\begin{prop} We have $\tens\bK_{m|n}=\bT_{m|n}$. Moreover, $\bK_{m|n}$ has an exhausting tensor filtration of length $\min(m,n)+1$.
\end{prop}

\begin{proof}  Obviously $\tens\bK_{m|n}\supset \bT_{m|n}$. Assume that  $\tens\bK_{m|n}\neq \bT_{m|n}$. Then since $\bT_{m|n}$ is injective in $\mathbb{T}_{\mathfrak g}$ the socle
of  $\tens\bK_{m|n}$ is larger than the socle of $\bT_{m|n}$, but this is a contradiction since $\soc \bT_{m|n}=\soc \bK_{m|n}$.
The second claim can be proven by induction on $\min(m,n)$, since $\bK_{m|n}/\bT_{m|n}$ is isomorphic to a submodule of $\bK^{\oplus mn}_{m-1|n-1}$ via the map $\oplus_{ij} ds_{ij}$.
\end{proof}

\subsection{Meaning of the socle filtration}
Now we will define a filtration on the category $\mathcal O_{m|n}^{\mathbb Z}$.
For a $\mathfrak {gl}(m|n)$-module $M$, let
$$X_M=\{x\in X_{\mathfrak{gl}(m|n)}\,|\,DS_x(M)\neq 0\},$$
and let $X^k_{\mathfrak{gl}(m|n)}$ be the subset of all elements in $X_{\mathfrak{gl}(m|n)}$ of rank less than or equal to $k$. We define
$[\mathcal O^{\mathbb Z}_{m|n}]^k$ to be the full subcategory of $\mathcal O^{\mathbb Z}_{m|n}$ consisting  of all modules $M$ such that
$X_M\subset X^k_{\mathfrak{gl}(m|n)}$. Note that $[\mathcal O^{\mathbb Z}_{m|n}]^k$ is not an abelian category. Furthermore,
we define $[\mathcal O^{\mathbb Z}_{m|n}]_{-}^k$ to be the full subcategory of $\mathcal O^{\mathbb Z}_{m|n}$ consisting of all modules $M$ such that
$$X_M\cap\mathfrak{gl}(m|n)_{-1}\subset X^k_{\mathfrak{gl}(m|n)}.$$

Let $\bK_{m|n}^k$ denote the complexification of the subgroup in $\bK_{m|n}$ generated by the classes of modules lying in $[\mathcal O^{\mathbb Z}_{m|n}]^k$,
and let $(\bK_{m|n}^k)_-$ be defined similarly for the category $[\mathcal O^{\mathbb Z}_{m|n}]_{-}^k$.
 Since both categories are invariant under the functors $\mathrm{E}_i$ and $\mathrm{F}_i$,  both $\bK_{m|n}^k$  and $(\bK_{m|n}^k)_-$  are
$\mathfrak{sl}(\infty)$-submodules of $\bK_{m|n}$.

\begin{conjecture} $\bK_{m|n}^k=\soc^{k+1}\bK_{m|n}$ and  $(\bK_{m|n}^k)_-=\tens^{k+1}\bK_{m|n}$.
\end{conjecture}
Here we prove a weaker statement.
Recall that $\mathcal O_{m|n}^{\mathbb Z}$ has block decomposition:
$$\mathcal O_{m|n}^{\mathbb Z}=\bigoplus(\mathcal O_{m|n}^{\mathbb Z})_{\chi},$$
where $(\mathcal O_{m|n}^{\mathbb Z})_{\chi}$ is the subcategory of modules admitting generalized central character $\chi$. The complexified reduced
Grothendieck group of $(\mathcal O_{m|n}^{\mathbb Z})_{\chi}$ coincides with the weight subspace $(\bK_{m|n})_{\chi}$. The degree of atypicality of $\chi$
is defined in \cite{DS}. In \cite{CS} it is proven that $(\mathcal O_{m|n}^{\mathbb Z})_{\chi}\subset [\mathcal O^{\mathbb Z}_{m|n}]^k$ if the degree of atypicality
of $\chi$ is not greater than $k$. Note that the degree of atypicality of the highest weight $\chi$ of the irreducible $\mathfrak{sl}_{\infty}$-module $\bV^{{\boldsymbol{\lambda},{\boldsymbol{\mu}}}}$
is equal to $m-|{\boldsymbol{\lambda}}|=n-|{\boldsymbol{\mu}}|$ and the degree of atypicality of any weight of $\bV^{{\boldsymbol{\lambda},{\boldsymbol{\mu}}}}$ is not
less than the degree of atypicality of the highest weight. Combining this observation with the description of the socle filtration of $\bK_{m|n}$ we obtain
the following.
  \begin{prop} $\soc^{k+1}\bK_{m|n}$ is the submodule in $\bK_{m|n}$ generated by weight vectors of weights with degree of atypicality less or
    equal to $k$. Therefore we have $\soc^{k+1}\bK_{m|n}\subset \bK_{m|n}^k$.
    \end{prop}


\section{Appendix}
In this section, we prove the technical lemma used in Lemma~\ref{lemma for thm}, which in turn is needed for the proof of Theorem~\ref{thm: T injective}.

Consider decompositions
$\bV=\bbW_1\oplus\bbW_2$ and $(\bV)_*=(\bbW_1)_*\oplus(\bbW_2)_*$
such that $\bbW_1^{\perp}=(\bbW_2)_*$ and $\bbW_2^{\perp}=(\bbW_1)_*$.
Denote by $\fs$
the subalgebra $\mathfrak{sl}(\bbW_1)$
of $\fg$. Let $\bT_{m|n}=\bV^{\otimes m}\otimes\bV_*^{\otimes n}$, and let
$\bY_{m|n}$  be the intersection with
$\bT_{m|n}$ of the ideal generated by
$\bbW_1\oplus(\bbW_1)_*$ in the tensor algebra $T(\bV\oplus\bV_*)$. Then $\bT_{m|n}$ considered as an $\fs$-module admits the decomposition
$$\Res_{\fs}\bT_{m|n}=(\bbW_2^{\otimes m}\otimes(\bbW_2)_*^{\otimes  n})\oplus \bY_{m|n}.$$

\begin{lem}\label{lem:tech} We have
  $$(\soc \bT_{m|n})\cap \bY_{m|n}\subset\fs \bY_{m|n}.$$
\end{lem}
\begin{proof} Note that $\bY_{m|n}$ is an object of $\widetilde{\mathbb  T}_{\fs}$ and
\begin{equation}\label{coinv}
\fs\bY_{m|n}=\bigcap_{\varphi\in\mathrm{Hom}_{\fs}(\bY_{m|n},\mathbb{C})} \ker\ \varphi.
\end{equation}
Let $\tau$ denote a map from $\{1,\dots,m+n\}$ to $\{1,2\}$. Denote by
$\bT_{m|n}^{\tau}$ the subspace of $\bT_{m|n}$ spanned by
$v_1\otimes\dots\otimes v_m\otimes u_{m+1}\otimes\dots\otimes u_{m+n}$
with $v_i\in \bbW_{\tau(i)}$ and $u_j\in (\bbW_{\tau(j)})_*$. Clearly,
$$\Res_{\fs}\bT_{m|n}=\bigoplus_{\tau}\bT^{\tau}_{m|n},$$
and we have an $\fs$-module isomorphism
$$\bT_{m|n}^{\tau}\cong \bbW_1^{\otimes p(\tau)}\otimes \bbW_2^{\otimes  (m-p(\tau))}\otimes (\bbW_1)_*^{\otimes q(\tau)}\otimes (\bbW_2)_*^{\otimes (n-q(\tau))},$$
where
$$p(\tau):=|\tau^{-1}(1)\cap\{1,\dots,m\}|, \quad q(\tau):=|\tau^{-1}(1)\cap\{m+1,\dots,m+n\}|.$$
Furthermore,
$$\bY_{m|n}=\bigoplus_{p(\tau)+q(\tau)>0}\bT^{\tau}_{m|n}.$$

Recall from \cite[Theorem 2.1]{PStyr} that
$$\soc \bT_{m|n}=\bigcap_{1\leq i\leq m, m<j\leq m+n} \ker\Phi_{ij},$$
where $\Phi_{ij}$ is defined in (\ref{contractions}). For $r=1,2$,  let
$\Phi^{\bbW_r}_{ij}:\bT_{m|n}\to \bT_{m-1|n-1}$ be defined by
$$ v_1\otimes\cdots\otimes v_m\otimes u_{m+1}\otimes\cdots\otimes
u_{m+n} \mapsto \langle u_j , v_i \rangle^{\bbW_r}
v_1\otimes\cdots\otimes \widehat{v_i} \otimes\cdots\otimes v_m\otimes u_{m+1}\otimes\cdots\otimes \widehat{u_j} \otimes\cdots\otimes u_{m+n},
$$
where $ \langle \cdot , \cdot \rangle^{\bbW_r}$ is defined on homogeneous elements by
$$ \langle u_j , v_i \rangle^{\bbW_r}:=\begin{cases} \langle u_j, v_i \rangle\,\,\text{if}\,\, u_j,v_i\in \bbW_r\\
0 \,\,\text{otherwise.}\end{cases}$$
Next, recall from \cite{DPS} that $\Hom_{\fs}(\bbW_1^{\otimes p}\otimes(\bbW_1)_*^{\otimes  q},\mathbb C)=0$ if $p\neq q$, and
if $p=q$, is spanned by compositions of contractions
$\Phi^{\bbW_1}_{1,j_1}\dots \Phi^{\bbW_1}_{p,j_p}$ for all possible
permutations $j_1,\dots,j_p$. Using (\ref{coinv}) we can conclude that
$\fs \bY_{m|n}^{\tau}=\bY_{m|n}^{\tau}$ if $ p(\tau)\neq q(\tau)$, whereas if
  $p=p(\tau)=q(\tau)$ we have
  $$\fs\bY_{m|n}^{\tau}=\bigcap_{i_1,\dots,i_p,j_1,\dots,j_p\in\tau^{-1}(1)}\ker\Phi^{\bbW_1}_{i_1,j_1}\dots \Phi^{\bbW_1}_{i_p,j_p}.$$
  Observe that
\begin{equation}\label{sum}
  \Phi_{ij}=\Phi^{\bbW_1}_{ij}+\Phi^{\bbW_2}_{ij}.
  \end{equation}

  We claim that if $y=\sum_{\tau}y_{\tau}\in \bY_{m|n}$ and $\Phi_{ij}(y)=0$
  for all $i,j$, then $y_{\tau}\in\fs \bT_{m|n}^{\tau}$ for all $\tau$. The
  statement is trivial for every $\tau$ such that $p(\tau)\neq q(\tau)$. Now we proceed to prove the claim in the case
  $p(\tau)=q(\tau)=p$ by induction on $p$.

Let $p=1$ and consider $\tau'$ with $p(\tau')=1=q(\tau')$. Let $i\leq m$ and $j>m$ be
such that $\tau'(i)=\tau'(j)=1$. Note that $\Phi_{i,j}(y_\tau')\in
(\bbW_2^{\otimes m-1}\otimes(\bbW_2)_*^{\otimes  n-1})$ and for
$\tau\neq\tau'$ we have $\Phi_{i,j}(y_{\tau})\in Y_{m-1|n-1}$. Therefore,
$\Phi_{i,j}(y_{\tau'})=\Phi_{i,j}^{\bbW_1}(y_{\tau'})=0$ and hence
$y_{\tau'}\in \fs \bT^{\tau'}_{m|n}$.

Now consider $y_{\tau'}$ such that $p(\tau')=p=q(\tau')$. Let  $i_1,\dots,i_p\leq m$ and $ j_1,\dots j_p>m$ such that $\tau'(i)=\tau'(j)=1$.
We would like to show that
\begin{equation}\label{eqn:tau} \Phi^{\bbW_1}_{i_1,j_1}\dots \Phi^{\bbW_1}_{i_p,j_p}(y_{\tau'})=\Phi_{i_1,j_1}\dots \Phi_{i_p,j_p}(y_{\tau'})=0.\end{equation}
Note that $\tau'$ has the property
\begin{equation}\label{property}\Phi_{i_1,j_1}\dots \Phi_{i_p,j_p}(y_{\tau'})\ \in\ \bbW_2^{\otimes m-p}\otimes(\bbW_2)_*^{\otimes  n-p}.\end{equation} Suppose that
$\tau''$ also has property (\ref{property}). Then
$(\tau'')^{-1}(1)\subset (\tau')^{-1}(1)$, and
if $\Phi_{i_1,j_1}\dots \Phi_{i_p,j_p}(y_{\tau''})\neq~0$, then
$\tau''(i_r)=\tau''(j_r)$ for all $r=1,\dots,p$. For every such
$\tau''\neq\tau'$ we have $p(\tau'')=q(\tau''):=l<p$. Let $\{i_{r_1},\dots,
i_{r_l},j_{r_1},\dots, j_{r_l}\}=(\tau'')^{-1}(1)$. Then by induction
assumption  $y_{\tau''}\in\fs \bT_{m|n}^{\tau''}$ and hence $$\Phi^{\bbW_1}_{i_{r_1},j_{r_1}}\dots \Phi^{\bbW_1}_{i_{r_l},j_{r_l}}(y_{\tau''})=\Phi_{i_{r_1},j_{r_1}}\dots \Phi_{i_{r_l},j_{r_l}}(y_{\tau''})=0.$$
But then
$$\Phi_{i_1,j_1}\dots \Phi_{i_p,j_p}(y_{\tau''})=0,$$
which implies
$$\Phi_{i_1,j_1}\dots \Phi_{i_p,j_p}(y_{\tau'})=0.$$
Now (\ref{eqn:tau}) follows, and this implies $y_{\tau'}\in\fs \bT_{m|n}^{\tau}$.
\end{proof}

\noindent Crystal Hoyt\\
Department of Mathematics, ORT Braude College \& Weizmann Institute, Israel\\
e-mail: crystal@braude.ac.il\\

\noindent Ivan Penkov\\
Jacobs University Bremen, Campus Ring 1, 28759, Bremen, Germany\\
e-mail: i.penkov@jacobs-university.de\\

\noindent Vera Serganova\\
Department of Mathematics, University of California Berkeley, Berkeley CA 94720, USA\\
e-mail: serganov@math.berkeley.edu\\

\end{document}